\documentclass{amsart}

\usepackage{amsthm}
\usepackage{amsmath}
\usepackage{graphicx}
\usepackage[margin=1.2in]{geometry}
\usepackage{amssymb,amsfonts}
\usepackage[all,arc]{xy}
\usepackage{enumerate}
\usepackage{mathrsfs}

\usepackage{tikz}

\DeclareMathOperator{\diam}{diam}

\newcommand{\rr}{\ensuremath{\mathbb{R}}}
\newcommand{\zz}{\ensuremath{\mathbb{Z}}}

\newcommand{\etanorm}[2]{\ensuremath{||#1||_{C^{0,\eta}(#2)}}}
\newcommand{\etasemi}[2]{\ensuremath{[#1]_{C^{0,\eta}(#2)}}}

\newcommand{\tlipsemi}[2]{\ensuremath{[#1]_{C_t^{0,1}(#2)}}}

\newcommand{\tetasemi}[2]{\ensuremath{[#1]_{C_t^{0,\eta}(#2)}}}

\newcommand{\bdry}[1]{\ensuremath{\partial_p #1}}
\newcommand{\linfty}[2]{\ensuremath{||#1||_{L^{\infty}(#2)}}}

\newcommand{\alphasemi}[2]{\ensuremath{[#1]_{C^{0,\alpha}(#2)}}}

\newcommand{\rdel}{\ensuremath{r_{\delta}}}

\newcommand{\ydel}{\ensuremath{\tilde{Y}_{\rdel}}}
\newcommand{\kdel}{\ensuremath{\tilde{K}_{\rdel}}}
\newcommand{\kdeltop}{\ensuremath{\tilde{K}_{\rdel}^T}}

\newcommand{\Umesh}{\ensuremath{\mathcal{U}_h}}

\newcommand{\xbar}{\ensuremath{\bar{x}}}
\newcommand{\tbar}{\ensuremath{\bar{t}}}

\newcommand{\tr}{\ensuremath{\text{tr}}}

\newcommand{\convmesh}{\ensuremath{v^-_{\theta,\theta}}}

\newcommand{\Ph}{\ensuremath{\mathcal{S}_{h}}}

\newcommand{\PM}{\ensuremath{\mathcal{P}_M}}
\newcommand{\PP}{\ensuremath{\mathcal{P}}}
\newcommand{\PPi}{\ensuremath{\mathcal{P}_\infty}}

\newcommand{\U}{\ensuremath{U}}

\newtheorem{thm}{Theorem}[section]
\newtheorem{cor}[thm]{Corollary}
\newtheorem{prop}[thm]{Proposition}
\newtheorem{lem}[thm]{Lemma}

\theoremstyle{definition}

\newtheorem{defn}[thm]{Definition}

\theoremstyle{remark}
\newtheorem{rem}[thm]{Remark}

\begin{document}
\title[Error estimates]{Error estimates for approximations of nonlinear uniformly parabolic equations}
\author{Olga Turanova}
\date{September 11, 2014}

\begin{abstract}
We introduce the notion of $\delta$-viscosity solutions for fully nonlinear uniformly parabolic PDE on bounded domains. We prove that $\delta$-viscosity solutions are uniformly close to the actual viscosity solution, with an explicit error of order $\delta^{\alpha}$. As a consequence we obtain an error estimate for implicit monotone finite difference approximations of uniformly parabolic PDE. 
\end{abstract}
\maketitle

\section{Introduction}
We introduce $\delta$-viscosity solutions for the nonlinear parabolic problem
\begin{equation}
\label{eq:main}
 u_t-F(D^2u)=0  \quad \text{ in } \Omega \times (0,T).
\end{equation}
We prove an estimate between viscosity solutions and $\delta$-viscosity solutions of (\ref{eq:main}) under the assumption that the nonlinearity $F$ is uniformly elliptic (see (F\ref{ellipticity}) below). As a consequence, we find a rate of convergence for monotone and consistent implicit finite difference approximations to (\ref{eq:main}). Both results generalize the work of Caffarelli and Souganidis in \cite{Approx schemes} and \cite{Rates Homogen}, who consider the time-independent case.

The nonlinearity $F$ is a continuous function on $\mathcal{S}_{n\times n} $, where $\mathcal{S}_{n\times n}$ is the set of $n\times n$ real symmetric matrices endowed with the usual order and norm (for $X\in \mathcal{S}_{n\times n}$,  $||X|| = \sup_{|v|=1}|Xv|$). We make the following assumptions:
\begin{enumerate}[({F}1)]
\item \label{ellipticity} $F$ is uniformly elliptic with constants $0<\lambda \leq \Lambda$, which means that for  any $X \in \mathcal{S}_{n\times n}$ and for all $Y\geq 0$,
\[
\lambda||Y|| \leq F(X+Y)-F(X) \leq \Lambda||Y||;
\]
and,
\item \label{last prop of F}  $F(0)=0$.
\end{enumerate} 
We also assume: 
\begin{enumerate}[({U}1)]
\item
\label{assum U}
$\Omega$ is a bounded subset of $\rr^n$.
\end{enumerate} 

The main result is an error estimate between viscosity solutions and $\delta$-viscosity solutions of (\ref{eq:main}). The definition of $\delta$-viscosity solutions is in Section \ref{subsection:delta def}. Next we present a statement of our main result that has been simplified for the introduction; the full statement is in Section \ref{section: comparison}.

\begin{thm}
\label{thm:vague comparison}
Assume (F\ref{ellipticity}), (F\ref{last prop of F}) and (U1). Assume $u$ is a viscosity solution of (\ref{eq:main}) that is Lipschitz continuous in $x$ and H\"older continuous in $t$. Assume $\{u_\delta\}_{\delta>0}$ is a family of H\"older continuous $\delta$-viscosity solutions of (\ref{eq:main}) that satisfy 
\[
u_\delta = u \text{ on the parabolic boundary of $\Omega \times (0,T)$ }
\]
 for all $\delta$. There exist positive constants $C$ and $\alpha$ that do not depend on $\delta$ such that  for all $\delta$ small enough,
\[
||u-u_\delta||_{L^{\infty}(\Omega \times (0,T))} \leq C\delta^\alpha .
\]
\end{thm}
We remark that if $u$ is the solution of a boundary value problem with sufficiently regular boundary data, then the regularity conditions on $u$ of Theorem \ref{thm:vague comparison} are satisfied (see Remark \ref{remAboutLips} for more details). We make this an assumption  in order to  avoid discussing boundary value problems and to keep the setting as simple as possible.

The notion of $\delta$-viscosity solutions was introduced by Caffarelli and Souganidis in \cite{Approx schemes}  as a tool for obtaining uniform estimates for viscosity solutions. In \cite{Approx schemes}, $\delta$-viscosity solutions were used to establish an error estimate for finite difference approximations of nonlinear uniformly elliptic equations. An error estimate between viscosity and $\delta$-viscosity solutions of uniformly elliptic equations was obtained by Caffarelli and Souganidis in \cite{Rates Homogen}. This was an important step in establishing a rate for homogenization in random media \cite{Rates Homogen}. 

The main challenge in obtaining an error estimate between viscosity and $\delta$-viscosity solutions, in both the elliptic and parabolic setting, is overcoming the lack of regularity of the viscosity solution $u$. Indeed, the proof of the error estimate in \cite{Approx schemes} is based on a regularity result \cite[Theorem A]{Approx schemes}, which says that outside of sets of small measure, solutions of uniformly elliptic equations have second-order expansions with controlled error. We prove a similar result for solutions of equation (\ref{eq:main}). This is Theorem \ref{Thm A} of this paper, and it is an essential part of our proof of Theorem \ref{thm:vague comparison}. 

But even once it is known that the solution $u$ of (\ref{eq:main}) has second-order expansions with controlled error, it is necessary to further regularize $u$ and  the $\delta$-viscosity solution $u_\delta$ in order establish an estimate for their difference. For this we use the classical inf- and sup- convolutions, along with  another regularization of inf-sup type, which we call  $x$-sup- and $x$-inf- convolutions. Both of these regularizations preserve the notions of viscosity and $\delta$-viscosity solution. Moreover, we show that the  regularity  we establish in Theorem \ref{Thm A} for solutions $u$ of (\ref{eq:main}) is also enjoyed by the $x$-inf and $x$-sup convolutions of $u$ (Proposition \ref{prop: more regularity for convolutions}).

In this paper we also study implicit finite difference approximations to (\ref{eq:main}) and prove an error estimate between the solution $u$ of (\ref{eq:main}) and approximate solutions to (\ref{eq:main}). We have simplified the notation in the introduction in order to state this result here; see Section \ref{section: approx} for all the details about approximation schemes and for  the full statement of the error estimate. 

We write the finite difference approximations as 
\begin{equation}
\label{eq:main h}
\Ph[u](x,t) = \delta^-_{\tau} u (x,t) - \mathcal{F}_h(\delta^2 u(x,t))=0  \text{ in } (\Omega \times (0,T))\cap E.
\end{equation}
Here $E=h\zz^n \times h^2\zz$ is the mesh of discretization, $\Ph[u]$ is the implicit finite difference operator, and $\delta^-_{\tau} u$ and $\delta^2 u$ are finite difference quotients associated to a function $u$.   We assume:
\begin{enumerate}[({S}1)] 
\item \label{monotone}  if $u_h^1$ and $u_h^2$ are solutions of (\ref{eq:main h}) with $u_h^1\leq u_h^2$ on the discrete boundary of $\Omega\times(0,T)$, then $u_h^1\leq u_h^2$ in $(\Omega\times(0,T))\cap E$; and
\item
\label{consistent} 
there exists a positive constant $K$ such that for all $\phi\in C^3(\Omega\times(0,T))$ and all $h>0$,
\[
\sup|\phi_t-F(D^2\phi)-\Ph[\phi]|\leq K(h+ h\linfty{D_x^3\phi}{\Omega\times (0,T)} + h^2\linfty{\phi_{tt}}{\Omega\times (0,T)}).
\]
\end{enumerate} 
Schemes that satisfy (S\ref{monotone}) and (S\ref{consistent}) are said to be, respectively, \emph{monotone} and \emph{consistent with an error estimate for $F$}.
 We prove:
\begin{thm}
\label{thm:mesh comparison vague}
Assume (F\ref{ellipticity}), (F\ref{last prop of F}) and (U1).  Assume that $\Ph$ is a monotone  implicit approximation scheme for (\ref{eq:main}) that is consistent with an error estimate. Assume  $u$ is a viscosity solution of (\ref{eq:main}) that is Lipschitz continuous in $x$ and H\"older continuous in $t$ and that $u_h$ is a H\"older continuous solution of (\ref{eq:main h}). Assume that for each $h>0$,
\[
u_h=u \text{ on the discrete boundary of $\Omega\times(0,T)$}. 
\]
 There exist positive constants $C$ and $\alpha$ that do not depend on $h$ such that for every $h$ small enough,
\[
||u-u_h||_{L^{\infty}((\Omega \times (0,T))\cap E)} \leq Ch^\alpha.
\]
\end{thm}

The convergence of monotone and consistent approximations of fully nonlinear second order PDE was first established by Barles and Souganidis \cite{BarlesSouganidis}. Kuo and Trudinger later studied the existence of monotone and consistent approximations for nonlinear elliptic and parabolic equations and the regularity of the approximate solutions $u_h$ (see \cite{KT estimates, Kuo Trudiner elliptic ops, Kuo Trudinger max princip parabolic, Kuo Trudinger parabolic diff ops}). They showed, both in the elliptic and in the parabolic cases, that if $F$ is uniformly elliptic, then there exists a monotone finite difference scheme $\Ph$ that is consistent with $F$, and that the approximate solutions $u_h$ are uniformly equicontinuous. However, obtaining an error estimate remained an open problem. 

The first error estimate for approximation schemes was established by Krylov in \cite{Krylov 1998} and \cite{Krylov 1999} for nonlinearities $F$ that are either convex or concave, but possibly degenerate. Krylov used stochastic control methods that apply in the convex  or concave case, but not in the general setting.  Barles and Jakobsen in \cite{BarlesJakobsen2002} and \cite{BarlesJakobsen2005} improved Krylov's error estimates for  convex or concave equations. In \cite{Krylov2005} Krylov improved the error estimate to be of order $h^{1/2}$, but still in the convex/concave case. In addition,  Jakobsen \cite{Jakobsen2004, Jakobsen2006} and Bonnans, Maroso, and Zidani \cite{Bonnans} established error estimates for special equations or for special dimensions. The first error estimate without a convexity or concavity assumption was obtained  Caffarelli and Souganidis in \cite{Approx schemes} for  nonlinear elliptic equations.  To our knowledge, Theorem \ref{thm:mesh comparison vague} is  the first  error estimate for  nonlinear parabolic equations that are neither convex nor concave.

While this paper was in preparation, we learned of the preprint of Daniel \cite{Daniel}. It contains a result \cite[Theorem 1.2]{Daniel} that is similar to our Theorem \ref{Thm A}. His proof also uses methods similar to those of Caffarelli and Souganidis in \cite[Theorem A]{Approx schemes}. Theorem \ref{Thm A} is the only overlap between this paper and \cite{Daniel}.

Our paper is structured as follows. In Section \ref{preliminaries} we establish notation, give the definition of $\delta$-viscosity solutions, and state several known results. In Section \ref{section: regularity} we prove Theorem \ref{Thm A}, the regularity result that is essential for the rest of our arguments. In Section \ref{section: Two different inf- and sup- convolutions} we discuss regularizations of inf-sup type and establish  Proposition \ref{prop: more regularity for convolutions}, a version of Theorem \ref{Thm A} for $x$-sup and $x$-inf convolutions. In Section \ref{section: main prop} we prove a key estimate between viscosity solutions of (\ref{eq:main}) and sufficiently regular $\delta$-viscosity solutions of (\ref{eq:main}). Theorems \ref{thm:vague comparison} and \ref{thm:mesh comparison vague} are straightforward consequences of this estimate. In Section \ref{section: comparison} we give the precise statement and proof of Theorem \ref{thm:vague comparison}. Section \ref{section: approx} is devoted to introducing the necessary notation and stating known results about approximation schemes. The precise statement and proof of the error estimate for approximation schemes is in Section \ref{section:error est approx}. We also include an appendix with known results about inf- and sup- convolutions.

\section{Preliminaries}
\label{preliminaries}
In this section we establish notation, recall the definition of viscosity solutions for parabolic equations, give the definition of $\delta$-viscosity solutions, and recall some known regularity results for solutions of (\ref{eq:main}).
\subsection{Notation}
Points in $\rr^{n+1}$ are denoted $(x,t)$ with $x\in \rr^n$. The parabolic distance between two points is 
\[
d((x,t), (y,s))=(|x-y|^2+|s-t|)^{1/2}.
\]
We denote the usual Euclidean distance in $\rr^{n+1}$ by $d_e((x,t),(y,s))= (|x-y|^2+|s-t|^2)^{1/2}$.
Throughout the argument we consider parabolic cubes, denoted by
\[
Q_r(x,t)=[x-r,x+r]^n\times(t-r^2,t],
\]
and forward and backward cylinders, 
\[
Y_r(x, t)= B_{r}(x) \times (t ,t+r^2] \text{ and } Y^-_r (x, t) = B_{r}(x) \times (t - r^2 ,t],
\]
where $B_r(x)=\{y: |x-y|<r\}$ is the open ball in $\rr^n$. We  write $B_r$, $Q_r$, $Y_r$ and $Y^-_r$ to mean $B_r(0)$, $Q_r(0)$, $Y_r(0)$ and $Y^-_r(0)$, respectively.

For $\Omega\subset \rr^n$, the parabolic boundary of $\Omega\times (a,b)$ is defined as
\[
\bdry (\Omega \times (a,b)) = (\Omega \times \{a\}) \cup (\partial \Omega \times (a,b)).
\]

The time derivative of a function $u$ is denoted by $u_t$ or $\partial_t u$. The gradient of $u$ with respect to the space variable $x=(x_1,...,x_n)$ is denoted by $Du=(u_{x_1},...,u_{x_n})$, and the Hessian of $u$ with respect to $x$ is denoted by $D^2u$. In addition, 
we use  $u^+$ and $u^-$ to denote, respectively, the positive and negative parts of a function $u$.

We say that a constant is \emph{universal} if it is positive and depends only on $\Lambda$, $\lambda,$ and $n$.

We introduce notation for several families of paraboloids. 
\begin{defn} Let $M>0$. We define:
\begin{enumerate}
\item the class of \emph{convex paraboloids of opening $M$}:
\[
\PM^+ =\left\{ P(x,t)=c+l\cdot x + mt + \frac{M}{2}|x|^2  \text{ where } l\in \rr^n, c, m\in \rr, |m| \leq M\right\};
\]
\item the class of \emph{concave paraboloids of opening $M$}:
\[
 \PM^- = \{-P(x,t) | P\in \PM^+\};
\] 
\item the class of paraboloids of arbitrary opening:
\[
\PPi = \left\{ P(x,t)=c+l\cdot x + m t + x\cdot Qx^T \text{ where } l\in \rr^n,c,m\in \rr, Q\in \mathcal{S}_{n\times n} \right \}; 
\]
\item the set of polynomials that are quadratic in $x$ and linear in $t$ with the only mixed term of the form $a\cdot xt$:
\[
\PP = \left\{ P(x,t)=c+l\cdot x + m t + a\cdot x t + x\cdot Qx^T \text{ where } l,a\in \rr^n, c, m\in \rr,Q\in \mathcal{S}_{n\times n} \right \}.
\]
\end{enumerate}
\end{defn}
\begin{rem}
\label{rem:pinfty} We remark that if $P(x,t)$ is a paraboloid in $\PPi$, then $P_t$ and $D^2P$ are constants.
\end{rem}
We will make use of  the following seminorms, norms, and function spaces:
\begin{defn}The class of continuous functions on $U\subset \rr^{n+1}$ is denoted $C(U)$. The class $C^2(U)$ is the set of functions $\phi$ that are differentiable in  $t$ and twice differentiable in $x$, with $\phi_t\in C(U)$ and $D^2\phi\in C(U)$.
\end{defn}
\begin{defn}
For $\eta\in (0,1]$ and $u\in C(\Omega\times (a,b))$,  we define:
\begin{align*}
\etasemi{u}{\Omega\times (a,b)} &= \sup_{(x,t), (y,s)\in \Omega\times (a,b)}\frac{|u(x,t)-u(y,s)|}{d((x,t),(y,s))^\eta},\\
[u]_{C^{0,\eta}_x(\Omega\times (a,b))}&= \sup_{x,y\in \Omega; t\in (a,b)} \frac{|u(x,t)-u(y,t)|}{|x-y|^\eta},\\
\tetasemi{u}{\Omega\times (a,b)} &= \sup_{ t,s\in (a,b); x\in \Omega} \frac{|u(x,t)-u(x,s)|}{|t-s|^\eta}, \text{ and,}\\
\etanorm{u}{\Omega\times (a,b)}&= \linfty{u}{\Omega\times (a,b)}+\etasemi{u}{\Omega\times (a,b)}.
\end{align*}
\end{defn}
\begin{defn} For $\eta\in (0,1]$, we define:
\begin{align*}
C^{0,\eta}(\Omega\times (a,b))&=\{ u\in C(\Omega\times(a,b)) :\etanorm{u}{\Omega\times (a,b)}<\infty\},\\
C^{0,\eta}_x(\Omega\times(a,b)) &= \left\{ u\in C(\Omega\times(a,b)): [u]_{C^{0,\eta}_x(\Omega\times(a,b))}+ \linfty{u}{\Omega\times (a,b)} <\infty\right\}.
\end{align*}
\end{defn}

We will also need the following notion of differentiability:
\begin{defn}
For a constant $K$, we say that $u$ satisfies $D^2u(x,t)\geq KI$ (resp. $D^2u(x,t)\leq KI$) \emph{in the sense of distributions} if there exists a quadratic polynomial 
\[
P(y)= c+p\cdot y +\frac{K|x-y|^2}{2}
\]
 such that $u(x,t)=P(x)$ and $u(y,t)\geq P(y)$ (resp. $u(y,t)\leq P(y)$) for all $y\in B_r(x)$, for some $r>0$.
\end{defn}

\subsection{Viscosity and $\delta$-viscosity solutions}
\label{subsection:delta def} 

We recall the definition of viscosity solutions for parabolic equations. 

\begin{defn}
\label{defn visc}
\begin{enumerate}
\item A function $u\in C(\Omega\times(0,T))$ is a \emph{viscosity supersolution} of (\ref{eq:main})  if for all $(x, t) \in \Omega\times(0,T)$, any $\phi\in C^2(\Omega\times (0,T))$ with $\phi \leq u $  on $Y^-_{\rho} (x, t)$ for some $\rho>0$ and  $\phi(x,t)=u(x,t)$ satisfies
\[
\phi_t(x,t) - F(D^2\phi(x,t)) \geq 0.
\]
\item A function $u\in C(\Omega\times(0,T))$ is a \emph{viscosity subsolution} of (\ref{eq:main}) if for all $(x, t) \in \Omega\times(0,T)$, any $\phi\in C^2(\Omega\times(0,T))$ with $\phi \geq u $  on $Y^-_{\rho} (x, t)$ for some $\rho>0$ and  $\phi(x,t)=u(x,t)$ satisfies
\[
\phi_t(x,t) - F(D^2\phi(x,t)) \leq 0.
\]
\item We say that $u\in C(\Omega\times(0,T))$ is a \emph{viscosity solution} of (\ref{eq:main}) if $u$ is both a sub- and a super- solution of (\ref{eq:main}).
\end{enumerate}
\end{defn}

\begin{rem}
In the above definitions, we require the test function $\phi$ to stay above (or below) $u$ on a backward cylinder $Y^-_\rho(x,t)$. However, this definition is equivalent to the usual one, which requires the test function to stay above (or below) $u$ on a Euclidean open set around $(x,t)$. This equivalence is proven by Crandall,  Kocan, and \'{S}wi\c{e}ch  \cite[Lemma 1.4]{crandallkocan} for $L^p$ viscosity solutions. Their proof carries over into our setting with  no modifications.
\end{rem}

See Section 8 of Crandall, Ishii, and Lions' \cite{User's guide} for further discussion of viscosity solutions of parabolic equations.

We now introduce $\delta$-viscosity solutions for  (\ref{eq:main}), following the definition of \cite{Approx schemes} and \cite{Rates Homogen}. 
\begin{defn}
\label{defn delta}
Fix $\delta>0$. 
\begin{enumerate}
\item A function $v\in C(\Omega\times(0,T))$ is a \emph{$\delta$-viscosity supersolution} of (\ref{eq:main}) if for all $(x, t) \in \Omega\times(0,T)$  such that $Y^-_{\delta}(x, t) \subset \Omega\times(0,T)$,  any $P\in\PPi$ with $P \leq v $  on $Y^-_{\delta} (x, t)$ and  $P(x,t)=v(x,t)$ satisfies
\[
P_t - F(D^2P) \geq 0.
\]
\item A function $v\in C(\Omega\times(0,T))$ is a \emph{$\delta$-viscosity subsolution} of (\ref{eq:main}) if for all $(x, t) \in \Omega\times(0,T)$  such that $Y^-_{\delta}(x, t) \subset \Omega\times(0,T)$,  any $P\in\PPi$ with $P\geq v$ on $Y^-_{\delta} (x, t)$ and  $P(x,t)=v(x,t)$ satisfies
\[
P_t - F(D^2P) \leq 0.
\]
\item
We say $v$ is a \emph{$\delta$-viscosity solution} of (\ref{eq:main}) if it is both a $\delta$-viscosity sub- and  super- solution.
\end{enumerate}
\end{defn}
From now on we will say ``solution" to mean viscosity solution and  ``$\delta$-solution" to mean $\delta$-viscosity solution.

From the definitions, it is clear that a viscosity solution of (\ref{eq:main}) is a $\delta$-solution of (\ref{eq:main}). The main difference between the definitions of viscosity and $\delta$-viscosity solution is that for $v$ to be a $\delta$-supersolution (resp. subsolution), any test polynomial must stay below (resp. above) $v$ on a set of fixed size.

\subsection{Known results}

We  introduce the Pucci extremal operators, which are defined for constants $0<\lambda\leq \Lambda$ and $X\in \mathcal{S}_{n\times n}$ by 
\begin{align*}
\mathcal{M}_{\lambda, \Lambda}^-(X) &= \lambda \sum_{e_i >0} e_i +\Lambda \sum_{e_i<0} e_i,\\ \mathcal{M}_{\lambda, \Lambda}^+(X) &= \Lambda \sum_{e_i >0} e_i +\lambda \sum_{e_i<0} e_i,
\end{align*}
where the $e_i$ are the eigenvalues of $X$ (see Caffarelli  and Cabre \cite{Cabre Caffarelli book} and Wang \cite{WangI}).
We also introduce the so-called upper and lower monotone envelopes of a function, which  play the role of the convex and concave envelopes in the regularity theory of elliptic equations. We follow the notes of  Imbert and Silvestre in \cite[Section 4]{ImbertSilvestre} for our presentation of Definition \ref{defn env}, Lemma \ref{lem:representation} and Proposition \ref{pre ABP} below.
\begin{defn}
\label{defn env}
Let $\Omega$ be a convex subset of $\rr^n$ and let  $u:\Omega \times (a,b) \rightarrow \rr$ be continuous.
\begin{enumerate}
\item The \emph{lower monotone envelope} of $u$ is the largest function $v:\Omega \times (a,b) \rightarrow \rr$ that lies below $u$ and is non-increasing with respect to $t$ and convex with respect to $x$. It is often denoted by $\underline{\Gamma}(u)$.
\item The \emph{upper monotone envelope} of $u$ is the smallest function $v:\Omega \times (a,b) \rightarrow \rr$ that lies above $u$ and is non-decreasing with respect to $t$ and concave with respect to $x$. It is often denoted by $\bar{\Gamma}(u)$.
\end{enumerate}
\end{defn}

We have the following representation formulas for the upper and lower envelopes of $u$ (\cite[Section 4]{ImbertSilvestre}):
\begin{lem}
\label{lem:representation}
Assume $u\in C(Y_\rho^-)$. Then
\[
\underline{\Gamma}(u)(x,t) = \sup \{\zeta \cdot x +h : \  \zeta \cdot y + h \leq u(y,s) \text{ for all } (y,s)\in  Y_\rho^-\cap\{s\leq t\}\}
\]
and
\[
\bar{\Gamma}(u)(x,t) = \inf \{\zeta \cdot x +h : \  \zeta \cdot y + h \geq u(y,s) \text{ for all } (y,s)\in  Y_\rho^-\cap\{s\leq t\}\}.
\]
\end{lem}

The following fact will play a central role in our arguments. In subsection \ref{subsec:proof of pre ABP} of the Appendix, we explain how Proposition \ref{pre ABP} follows from parts of the proof of   the parabolic version of the Alexandroff-Bakelman-Pucci estimate as presented in \cite[Section 4.1.2]{ImbertSilvestre}. We refer the reader to \cite[Section 4]{ImbertSilvestre} for a history of the  parabolic version of the Alexandroff-Bakelman-Pucci estimate and the relevant references.
\begin{prop}
\label{pre ABP}
Assume $u\in C(Y^-_\rho)$ is such that $u\geq 0$ on $\bdry Y^-_\rho$. 
Assume that  there exists a constant $K$ so that $\tlipsemi{u}{Y^-_\rho}\leq K$ and $D^2u(x,t)\leq  KI$ in the sense of distributions for every $(x,t)\in Y^-_\rho$. 
 There exists a universal constant $C$ such that
\begin{equation}
\label{cons of ABPbd}
\sup_{Y^-_\rho} u^- \leq C \rho^{\frac{n}{n+1}} |\{u=\Gamma\}|^{\frac{1}{n+1}}K,
\end{equation}
where  $\Gamma$ is the lower monotone envelope of $\min(u,0)$  extended by 0 to $Y^-_{2\rho}$.
\end{prop}

Next we state several regularity results. We will use the following interior H\"older gradient estimate  \cite[Theorems 4.8 and 4.9]{WangII}.

\begin{thm}
\label{thm:Wang's c alpha thm}
Assume  (F\ref{ellipticity}) and (F\ref{last prop of F}). There exist universal constants $\alpha\in (0,1)$ and $C$ such that if $u\in C(Q_1)$  is a solution of $u_t-F(D^2u)=0$ in $Q_1$, then $u_t\in C^{0,\alpha}(Q_1)$ and $Du\in C^{0,\alpha}(Q_1)$  with
\[
||Du||_{C^{0,\alpha}(Q_{1/2})} + ||u_t||_{C^{0,\alpha}(Q_{1/2})} \leq C(\linfty{u}{Q_1} + 1).
\]
\end{thm}
The following proposition follows from Theorem \ref{thm:Wang's c alpha thm} by a standard rescaling and covering argument.
\begin{prop}
\label{prop:diff of soln}
Assume  (F\ref{ellipticity}) and (F\ref{last prop of F}). Assume that  $V$ is a subset of $\Omega \times (0,T)$ with $d(V, \bdry \Omega \times (0,T))>r$. There exist universal constants $\alpha\in (0,1)$ and $C$ such that if $u\in C(\Omega\times (0,T))$  is a solution of $u_t-F(D^2u)=0$ in $\Omega\times (0,T)$, then
\begin{align*}
r^{1+\alpha}\alphasemi{D u}{V} + r\linfty{D u}{V} &\leq C(\linfty{u}{\Omega \times (0,T)} + 1)\\
r^{2+\alpha}\alphasemi{u_t}{V} + r^2\linfty{u_t}{V} &\leq C(\linfty{u}{\Omega \times (0,T)} + 1).
\end{align*}
\end{prop}

\section{The regularity result}
\label{section: regularity}
In this section, we establish the parabolic version of the regularity result  \cite[Theorem A]{Approx schemes}, which says that outside of sets of small measure, solutions of uniformly elliptic equations have second-order expansions with controlled error.  

\begin{defn}
Given $u\in C(Y_1)$, we define the set $\Psi_M (u,Y_1)\subset Y_1$ by
\begin{equation*}
\begin{split}
\Psi_M (u,Y_1) =\{
 &(x,t) \in Y_1  :\text{ there exists }  P\in \PP \text{ such that for all }(y,s)\in Y_1\cap \{s\leq t\},
\\ &|u(y,s) - u(x,t) - P(y,s)|\leq  \\
&\   \   \leq nM (|x-y|^3 + |x-y|^2|t-s| +|x-y||t-s|+|t-s|^2)
\}.
\end{split}
\end{equation*}
\end{defn}
In order to state our result, we introduce another family of subsets of $\rr^{n+1}$. We define
\[
K_r(x,t)=\left[x-\frac{r}{9\sqrt{n}}, x+\frac{r}{9\sqrt{n}}\right]^n \times\left(t , t+ \frac{r^2}{81n}\right].
\]
We denote $K_r(0,0)$ by $K_r$.

We  prove the following growth estimate for $|\Psi_M(u,Y_1)|$: 

\begin{thm}
\label{Thm A}
Assume  $F$ satisfies (F\ref{ellipticity}) and (F\ref{last prop of F}). Assume $u$ is a solution of  $u_t-F(D^2u)=0$ in  $Y_1$ and  $Du$ and $u_t$ exist and are continuous in $Y_1$. There exist universal constants $C$, $M_0$ and $\sigma$ such that for every $M \geq M_0$,
\begin{equation}
\label{eq:size bad set}
| K_1\setminus \Psi_M(u,Y_1)| \leq CM^{-\sigma}\left(\linfty{D u}{Y_1}^{\sigma} +\linfty{u_t}{Y_1}^{\sigma}\right).
\end{equation}
\end{thm}

\begin{rem}
If $u$ is a solution of $u_t-F(D^2u)=0$ in $\Omega\times (0,T)$ and $Y_1$ is compactly contained in $\Omega\times (0,T)$, then $u$ satisfies the hypotheses of Theorem \ref{Thm A}  because, according to Proposition \ref{prop:diff of soln}, $Du$ and $u_t$ exist and are continuous in $Y_1$.
\end{rem}
\begin{rem}
 Suppose $(x,t)\in \Psi_M (u,Y_1)$ and $P$ is the paraboloid given by the definition of $\Psi_M (u,Y_1)$. If $u$ is a viscosity solution of $u_t-F(D^2u)=0$ in $Y_1$ , then of course $P_t(x,t) - F(D^2P(x,t))=0$. On the other hand, suppose $u$ is only a $\delta$-solution of $u_t-F(D^2u)=0$ in $Y_1$, and $(x,t)\in \Psi_M (u,Y_1)$ with $Y^-_\delta(x,t)\subset Y_1$. In this case we \emph{cannot} conclude that $P_t(x,t) - F(D^2P(x,t))=0$. This is because in the definition of $\Psi_M (u,Y_1)$ we consider paraboloids $P$ in the class $\PP$, while in the definition of $\delta$-solution we allow only paraboloids in the smaller class $\PPi$. In Corollary \ref{cor} we establish a version of Theorem \ref{Thm A} with $P\in \PPi$ instead of $P\in\PP$.
\end{rem}

An essential element of our proof is \cite[Theorem 4.11]{WangI}, which says that solutions of uniformly parabolic equations have \emph{first} order expansions with controlled error on large sets (see also \cite[Section 3]{Daniel}). To state this result, we recall the following definitions:
\begin{align*}
\overline{G}_M (u, Y_1)  = \{ &(x,t) \in Y_1 :\text{ there exists }P\in\PP^+_M \\& \text{ with } P(x,t)=u(x,t) \text{ and $P \geq u$ on $Y_1\cap \{s \leq t\}$} \}, 
\\ \underline{G}_M (u, Y_1) & = \overline{G}_M (-u, Y_1), \text{ and } \\G_M(u, Y_1) &=  \left. \overline{G}_M (u, Y_1) \cap \underline{G}_M(u, Y_1). \right.
\end{align*}
We observe that if $(x,t)\in G_M(u,Y_1)$, then there exists  $p\in \rr^n$ such that for all $(y,s)\in Y_1\cap \{s\leq t\}$,
\begin{equation*}
|u(y,s) - u(x,t) - p\cdot (y-x) | \leq \frac{1}{2}M|x-y|^2+M|t-s|.
\end{equation*}
We now present \cite[Theorem 4.11]{WangI}, with notation adapted for our setting:
\begin{thm}
\label{thm:Wang's good set thm}
Assume  (F\ref{ellipticity}) and (F\ref{last prop of F}). Assume that $v$ is a solution of 
\begin{equation}
\label{class S}\left\{
\begin{array}{l l}
v_t- \mathcal{M}^-(D^2v) \geq 0, &\quad \\
v_t- \mathcal{M}^+(D^2v) \leq 0 &\quad
\end{array}\right.
\end{equation}
in $Y_1$. There exist universal constants $C$, $\sigma$, and $M_0$ such that for any $M \geq M_0$,
\[
|K_1 \setminus G_M(v, Y_1)| \leq C M^{-\sigma} (\linfty{v}{Y_1})^{\sigma}.
\]
\end{thm}

The main idea of the proof of Theorem \ref{Thm A} is to apply the growth estimate  of Theorem \ref{thm:Wang's good set thm} to the derivatives of $u$. Formally, upon differentiating $u_t-F(D^2u)=0$, we obtain that the derivatives $u_{x_i}$ and $u_t$ solve the linear parabolic equation $\partial_t u_{x_i} - \tr (DF \cdot D^2u_{x_i})=0$. Therefore, the estimates of Theorem \ref{thm:Wang's good set thm} apply to $u_{x_i}$ and $u_t$, and from this the estimates on $u$ are deduced. To carry out this plan, we will first use the following proposition \cite[Theorem 4.6]{WangII}.
\begin{prop}
\label{diff eqn}
Assume that $u$ is a solution of $u_t-F(D^2u)=0$ in $Y_1$. Then $u_t$ and $u_{x_i}$, for $i=1,..,n$, satisfy (\ref{class S}) in $Y_1$.
\end{prop}
Next we proceed with the proof of Theorem \ref{Thm A}. We remark that a result similar to \cite[Theorem A]{Approx schemes} was established later, via similar methods, by Armstrong, Silvestre and Smart \cite[Section 5]{ArmSmartSilv}. Our proof is based on the arguments in  \cite[Theorem A]{Approx schemes} and our presentation follows that of \cite[Section 5]{ArmSmartSilv}.

\begin{proof}[Proof of Theorem \ref{Thm A}]

By Proposition \ref{diff eqn},  the derivatives $u_{x_i}$ for $i=1,..., n$ and $u_t$ satisfy (\ref{class S}) in $Y_1$. Therefore, by Theorem \ref{thm:Wang's good set thm}, for any $M\geq M_0$ the size of the ``bad sets" of the $u_{x_i}$ and of $u_t$ are controlled: we have
\begin{equation}
\label{bd size sets}
\begin{split}
|K_1 \setminus G_{M}(u_{x_i}, Y_1)| \leq C {M}^{-\sigma} \linfty{Du}{Y_1} ^{\sigma} \text{ for  } i=1,..,n,\\
|K_1 \setminus  G_{M}(u_{t}, Y_1)| \leq C {M}^{-\sigma} \linfty{u_t}{Y_1}^{\sigma},
\end{split}
\end{equation}
where $C$ is a universal constant. We define the set $G_M$ to be the intersection of the  ``good sets"  of the derivatives of $u$: 
\[
G_M=\cap_{i=1}^{n} G_{M}(u_{x_i}, Y_1) \cap G_{M}(u_{t}, Y_1).
\]
The estimates of (\ref{bd size sets}) imply the following bound for the size of complement of $G_M$:
\begin{equation}
\label{sizeGm}
|K_1\setminus G_M|\leq   C M^{-\sigma} \left(\linfty{Du}{Y_1}^{\sigma} + \linfty{u_t}{Y_1}^{\sigma}\right),
\end{equation}
where $C$ is a universal constant. To prove that the desired upper bound (\ref{eq:size bad set}) on the size of the complement of $\Psi_M(u,Y_1)$ holds, it will be enough to establish the following relationship between $G_M$ and $\Psi_M(u,Y_1)$:
\begin{equation}
\label{GminPsi}
(G_M\cap K_1)\subset (\Psi_M(u,Y_1)\cap K_1).
\end{equation}
Indeed, we may use  (\ref{GminPsi}) and then the estimate  (\ref{sizeGm}) to obtain: 
\[
|K_1\setminus \Psi_M(u,Y_1)| \leq |K_1\setminus G_M|\leq  C M^{-\sigma} \left(\linfty{Du}{Y_1}^{\sigma} + \linfty{u_t}{Y_1}^{\sigma}\right).
\]
Let us now prove that (\ref{GminPsi}) holds. To this end, fix $(x,t) \in G_M \cap K_1$. The goal is to show  $(x,t) \in \Psi(u,Y_1)\cap K_1$; in other words, to produce a polynomial $P$ that is the second order expansion of $u$ at $(x,t)$. 

Because $(x,t) \in G_M$, there exist $p^1,...,p^{n+1}\in \rr^n$  such that for $i=1,..., n$ and for all $(y,s)\in Y_1\cap \{s\leq t\}$,
\begin{equation}
\label{eq:u_xi -p^i}
\begin{split}
|u_{x_i}(y,s) - u_{x_i} (x,t) - p^i\cdot (y-x) | &\leq \frac{1}{2}M|x-y|^2+M|t-s|, \text{ and}\\
|u_t(y,s) - u_t (x,t) - p^{n+1}\cdot (y-x) | &\leq \frac{1}{2}M|x-y|^2+M|t-s|.
\end{split}
\end{equation}
We point out that $D^2u(x,t)=(p^1,...,p^n)$ and $Du_t(x,t)=p^{n+1}$. We define  the $n\times n$ matrix $Q$ by $Q_{ij}=p_i^j$ and define the paraboloid $P\in \PP$ by
\begin{equation}
\label{defnP}
P(y,s) =  Du(x,t)\cdot (y-x) +\frac{1}{2}(y-x)\cdot Q(y-x)^T + (s-t)u_t(x,t)+\frac{1}{2}p^{n+1}\cdot (y-x) (s-t).
\end{equation}
We will now show that this polynomial $P$ is  the second order expansion of $u$ at $(x,t)$; more precisely, we will prove
\begin{equation}
\label{est btwn u P}
\begin{split}
|u(y,s)-u(x,t)-P(y,s)|\leq & nM ( |x-y|^3 + |x-y|^2|t-s| +|x-y||s-t|+|t-s|^2) \\
&\text{ for all }(y,s)\in Y_1\cap\{s\leq t\} .
\end{split}
\end{equation}
Once (\ref{est btwn u P}) is established, we may conclude $(x,t) \in \Psi(u,Y_1)\cap K_1$, and the proof of the theorem will be complete. 

To prove (\ref{est btwn u P}), let us fix any $(y,s)\in Y_1\cap\{s\leq t\}$. First, we express the difference between $u(y,s)$ and $u(x,t)$ in the following way:
\[
u(y,s)-u(x,t)=\int_0^1 \frac{d}{d\tau} u(x+\tau(y-x),t + \tau(s-t)) \, d\tau.
\]
Expanding the right-hand side of the previous line in terms of $Du$ and $u_t$ we find
\begin{equation*}
\begin{split}
u(y,s) -u(x,t) = & \int_0^1  (y-x)\cdot Du(x+\tau(y-x),t + \tau(s-t))\, d\tau + \\
&\    + \int_0^1 (s-t) u_t(x+\tau(y-x),t + \tau(s-t)) d\tau.
\end{split}
\end{equation*}
Now we subtract $P(y,s)$ from both sides of the previous line. Using the definition of $P$ (equation (\ref{defnP})) and rearranging yields,
\begin{equation*}
\begin{split}
u(y,s)&-u(x,t)- P(y,s)= \\
& (y-x)\cdot \int_0^1  \left(Du(x+\tau(y-x),t + \tau(s-t))-Du(x,t)-\tau Q(y-x)^T\right)\, d\tau + \\
&+ (s-t)\int_0^1 u_t(x+\tau(y-x),t + \tau(s-t))-u_t(x,t)-\tau p^{n+1}\cdot (y-x) \, d\tau.
\end{split}
\end{equation*}
We take the absolute value of both sides of the previous line and find
\begin{equation*}
\begin{split}
|u(y,s)& -u(x,t)-P(y,s)|\leq \\
& |y-x| \int_0^1  \left(\sum_{i=1}^n  |u_{x_i}(x+\tau(y-x),t + \tau(s-t))-u_{x_i}(x,t) - \tau p^i \cdot (y-x)|^2 \, d \tau \right)^{1/2}+ \\
&+ |s-t|\int_0^1 |u_t(x+\tau(y-x),t + \tau(s-t))-u_t(x,t)-\tau p^{n+1}\cdot (y-x) |\, d\tau.
\end{split}
\end{equation*}
We  use (\ref{eq:u_xi -p^i}) to bound each term on the right-hand side of the previous line, and obtain
\begin{align*}
|u(y,s)-u(x,t)-P(y,s)|\leq  & n |y-x|\int_0^1 \frac{1}{2}M\tau^2|x-y|^2 + M\tau|t-s| \, d\tau +\\& \   \    |s-t|\int_0^1 \frac{1}{2}M\tau^2|x-y|^2 + M\tau|t-s| \, d\tau. 
\end{align*}
Integrating in $\tau$ and combining terms yields 
\[
|u(y,s)-u(x,t)-P(y,s)|\leq nM ( |x-y|^3 +|x-y||s-t|+ |x-y|^2|t-s| +|t-s|^2) .
\]
Since $(y,s)\in Y_1\cap \{s\leq t\}$ was arbitrary, we have established (\ref{est btwn u P}) and so the proof of the theorem is complete.
\end{proof}

\begin{cor}
\label{cor}
Assume  $F$ satisfies (F\ref{ellipticity}) and (F\ref{last prop of F}). Assume $u$ is a solution of  $u_t-F(D^2u)=0$ in $Y_r(\bar{x},\bar{t})$ and  $Du$ and $u_t$ exist and are continuous in $Y_r(\bar{x},\bar{t})$.
There exists a subset $\Psi_M(u,Y_r(\bar{x},\bar{t}))$ of $Y_r(\bar{x},\bar{t})$ and a universal constant $C$ such that for any $(x,t)\in \Psi_M(u,Y_r(\bar{x},\bar{t}))$, there exists $Q\in \PPi$ with
\[
Q_t-F(D^2Q)=0
\]
and such that on $Y_r(\bar{x},\bar{t})\cap \{s\leq t\}$,
\begin{equation}
\label{eq:u-p cor}
\begin{split}
 |u(y,s) - u(x,t) - Q(y,s)| \leq \frac{nM}{r^2} &( r|x-y|^3 + |x-y|^2|t-s| +|t-s|^2) +\\&+\frac{C}{r}(M+\linfty{u_t}{Y_r(\bar{x},\bar{t})})|x-y|(t-s).
 \end{split}
\end{equation}
Moreover,
\begin{equation}
\label{better bdn on Theta_M}
| K_r(\bar{x},\bar{t})\setminus \Psi_M(u,Y_r(\bar{x},\bar{t}))| \leq \frac{Cr^{n+2}}{M^{\sigma}}(r^{-\sigma}\linfty{Du}{Y_r(\bar{x},\bar{t})}^\sigma+\linfty{u_t}{Y_r(\bar{x},\bar{t})}^\sigma).
\end{equation}
\end{cor}
\begin{proof}
We prove the statement for $r=1$ and $(\bar{x},\bar{t})=(0,0)$. The general case follows by translation and rescaling:  the equation that $u$ satisfies is translation invariant, and given $u$ that solves $u_t - F(D^2u) = 0$
in $Y_r$, the rescaled function 
$
\hat{u}(\hat{x},\hat{t})=\frac{1}{r^2} u(r\hat{x}, r^2 \hat{t})
$
solves the same equation in $Y_1$. We omit the details of the rescaling argument and proceed with the proof in the case $r=1$ and $(\bar{x},\bar{t})=(0,0)$. 

Fix $(x,t)\in \Psi_M(u,Y_1)$. We define the paraboloid $Q\in \PPi$ by $Q(y,s)=P(y,s) - \frac{1}{2}p^{n+1}\cdot (y-x)(s-t)$, where $P$ and $p^{n+1}$ are as in the proof of Theorem \ref{Thm A}. By the definition of $Q$, we have,
\[
|u(y,s) - u(x,t) - Q(y,s)| \leq |u(y,s) - u(x,t) - P(y,s)| +\frac{1}{2}|p^{n+1}| |y-x||s-t|.
\]
We use (\ref{est btwn u P}) to bound from above the first term on the right-hand side of the previous line, and find, for all $(y,s)\in Y_1\cap \{s\leq t\}$,
\begin{equation}
\label{eqn:uQ}
\begin{split}
 |u(y,s) - u(x,t) - Q(y,s)| \leq nM&( |x-y|^3 +|x-y||s-t|+ |x-y|^2|t-s| +|t-s|^2) +\\&+\frac{1}{2}|p^{n+1}| |y-x||s-t|.
\end{split}
\end{equation}
Let us now estimate $|p^{n+1}|$. To this end, we use the second inequality in (\ref{eq:u_xi -p^i}) to obtain, for any $(y,s) \in Y_1\cap \{s\leq t\}$,
\[
|p^{n+1}\cdot (y-x)| \leq \frac{1}{2}M|x-y|^2+M|t-s|+2\linfty{u_t}{Y_1}.
\]
Let us evaluate the previous inequality at $s=t$ and  at a point $y$ such that $|x-y|=\frac{1}{2}$. This yields,
\[
|p^{n+1}|\leq \frac{1}{4}M+4\linfty{u_t}{Y_1}.
\]
We use this to bound the last term on the right-hand side of (\ref{eqn:uQ}). We also rearrange so that both  terms involving $|y-x||s-t|$ are together. This yields the bound (\ref{eq:u-p cor}) with $r=1$ and $(\bar{x},\bar{t})=(0,0)$, and so the proof is complete. 
\end{proof}

\section{Two regularizations of inf- sup- type}
\label{section: Two different inf- and sup- convolutions}
This section is devoted to introducing and establishing the properties of two regularizations of inf-sup type. In subsection \ref{subsec:infsup} we recall the definitions of inf- and sup- convolutions in both the space and time variables, which  are quite similar to those used in the regularity theory of elliptic equations, and to the regularizations of  \cite[Section 4]{WangII}. In the proof of our main result, Theorem \ref{thm:vague comparison},  we will use inf- and sup- convolutions to regularize the $\delta$-solution $u_\delta$.

In subsection \ref{subsec:xinf} we define  regularizations of inf-sup type that we call  $x$-inf and $x$-sup convolutions. We will use
them to regularize the viscosity solution $u$ in the proofs of Theorems \ref{thm:vague comparison} and  \ref{thm:mesh comparison vague}. 

In  subsection \ref{subsec:regxinf}, we prove that if $u$ is a solution of $u_t-F(D^2u)=0$ then  the $x$-inf and $x$-sup convolutions of $u$ enjoy  regularity properties similar to those established in Theorem \ref{Thm A} for $u$. This result is a very important ingredient in our proofs of Theorems \ref{thm:vague comparison} and \ref{thm:mesh comparison vague}. 
\subsection{Regularization in both the space and time variables}
\label{subsec:infsup}
We recall the definitions of inf- and sup- convolutions. 
We use the notation $v^-_{\theta, \theta}$, instead of the expected $v^-_{\theta}$, to distinguish these from the  $x$-inf- and $x$-sup- convolutions that we introduce in the next subsection.
\begin{defn}
For $v\in C(\Omega \times (0,T))$ and $\theta>0$, we define the \emph{inf-convolution} $v^-_{\theta, \theta}$ and the  \emph{sup-convolution} $v^+_{\theta, \theta}$ by
\[
v^-_{\theta, \theta}(x,t) = \inf_{\Omega\times(0,T)} \left\{ v(y,s) + \frac{|x-y|^2}{2\theta}+\frac{|t-s|^2}{2\theta}\right\}
\]
and
\[ 
v^+_{\theta, \theta}(x,t) = \sup_{\Omega\times(0,T)} \left\{ v(y,s) - \frac{|x-y|^2}{2\theta}-\frac{|t-s|^2}{2\theta}\right\}.
\]
\end{defn}
\begin{defn}
\label{defnUv}
Given $\theta>0$, $\delta>0$ and $v\in C(\Omega \times (0,T))$,  we define
\[
U^{\theta,\delta} = \left\{(x,t)\in \Omega\times(0,T):\  \   d_e((x,t),\bdry(\Omega \times (0,T))) \geq 2\theta^{1/2} \linfty{v}{\Omega \times (0,T)}^{1/2} + \delta\right\}.
\]
\end{defn}
We summarize the basic properties of inf- and sup- convolutions in Proposition \ref{properties of inf-convolution} of the Appendix.  One property that is particularly important and useful to us is that taking inf- and sup- convolution preserves the notion of $\delta$-super and $\delta$-sub solutions. We state this precisely in item (\ref{inf conv is delta soln}) of  Proposition \ref{properties of inf-convolution}.

\subsection{Regularization in only the space variable}
\label{subsec:xinf}
We now define the $x$-inf and $x$-sup convolutions.
\begin{defn}
\label{defxinf}
For $u\in C(\Omega \times (0,T))$ and $\theta>0$, we define the \emph{$x$-sup convolution} $u^+_{\theta}$ and  \emph{$x$-inf-convolution} $u^-_{\theta}$ of $u$ by
\[
u^+_{\theta}(x,t) = \sup_{y\in \Omega} \left\{ u(y,t) - \frac{|x-y|^2}{2\theta}\right\} \text{ and }   u^-_{\theta}(x,t) = \inf_{y\in \Omega} \left\{ u(y,t) + \frac{|x-y|^2}{2\theta}\right\}.
\]
\end{defn}
\begin{defn}
\label{defnUtheta}
For $u\in C^{0,1}_x(\Omega \times (0,T))$ and $\tilde{\Omega}\subseteq \Omega$ we define the subset $\tilde{\Omega}^{\theta}$ of $\tilde{\Omega}$ by
\[
\tilde{\Omega}^{\theta} = \left\{x\in \tilde{\Omega}  :\  \    \inf_{y\in \partial \tilde{\Omega} } |x-y| \geq 2\theta \linfty{Du}{\Omega \times (0,T)}\right\}.
\]
In addition, for $U=\tilde{\Omega}\times I$ where $I\subseteq(0,T)$ is an interval, we define the subset $U^\theta$ of $U$ by
\[
U^{\theta}=\tilde{\Omega}^\theta\times I.
\]
\end{defn}
Note that for the definition of $U^\theta$ to be meaningful, the size of $U$ cannot be too small compared to $\theta$. Whenever we use this notation we make sure this is not the case. 

In addition, we denote $Y^{\theta}_r(\bar{x},\bar{t}) =(Y_r(\bar{x},\bar{t}) )^{\theta}$ and $K^{\theta}_r(\bar{x},\bar{t}) =(K_r(\bar{x},\bar{t}) )^{\theta}$. We will be using these families of sets quite frequently, so  we write down their definitions explicitly for the convenience of the reader:
\begin{align*}
Y^{\theta}_r(\bar{x},\bar{t}) &= B(\bar{x}, r-2\theta\linfty{D u}{\Omega \times (0,T)})\times (\bar{t}, \bar{t}+r^2]\\
K^{\theta}_r (\bar{x},\bar{t}) &= \left[\bar{x}-\left(\frac{r}{9\sqrt{n}}-2\theta\linfty{D u}{\Omega\times(0,T)}\right), \bar{x}+\left(\frac{r}{9\sqrt{n}}-2\theta\linfty{D u}{\Omega\times(0,T)}\right)\right]^n    \times\left(\bar{t}, \bar{t}+\frac{r^2}{81n}\right].
\end{align*}

In the following proposition we state the facts about $x$-inf- and $x$-sup- convolutions that we will use in this paper. Their proofs are very similar to those in the elliptic case  (see \cite[Chapter 4]{Rates Homogen} or \cite[Proposition 5.3]{Cabre Caffarelli book}) and we omit them.  
\begin{prop}
\label{prop:x-sup-conv}
Assume $u\in C^{0,1}_x(\Omega \times (0,T))$.
Then:
\begin{enumerate}
\item  \label{item:Utheta}  If $(x^*, t)$ is any point at which the infimum (resp. supremum) is achieved in the definition of $u^-_{\theta}(x,t)$ (resp. $u^+_{\theta}(x,t)$), then 
\[|x-x^*|\leq 2\theta\linfty{Du}{\Omega \times (0,T)}.
\]
Moreover, if $(x,t)\in U^\theta$ then $(x^*, t)\in U$. 
\item  We have $u^-_{\theta}(x,t)\leq u(x,t)\leq u^+_{\theta}(x,t)$ for all $(x,t)\in \Omega\times (0,T)$.
\item 
\label{item: bound from above on sup conv}
For all $(x,t)\in \Omega^\theta\times (0,T)$, we have
\[
u^+_{\theta}(x,t) -2\theta \linfty{Du}{\Omega \times (0,T)}^2\leq u(x,t)\leq 
u^-_{\theta}(x,t)+ 2\theta \linfty{Du}{\Omega \times (0,T)}^2 .
\]
\item 
\label{item:sup conv is semiconvex}
In the sense of distributions, $D^2 u^+_{\theta}(x,t) \geq -\theta^{-1}I$ and $D^2 u^-_{\theta}(x,t) \leq \theta^{-1}I$ for all $(x,t)\in \Omega\times (0,T)$.
\item 
\label{item: utheta is subsolution} 
If  $u$ is a subsolution of
\begin{equation}
\label{eqnc}
u_t-F(D^2u)=c
\end{equation} 
in $\Omega \times (0,T)$, then $u^+_{\theta}$ is a subsolution of (\ref{eqnc}) in $\Omega^{\theta}\times (0,T)$.
If  $u$ is a supersolution of (\ref{eqnc}) in $\Omega \times (0,T)$, then $u^-_{\theta}$ is a supersolution of (\ref{eqnc}) in $\Omega^{\theta}\times (0,T)$.
\end{enumerate}
\end{prop}

\subsection{Regularity of  $x$-inf and $x$-sup convolutions}
\label{subsec:regxinf}
This subsection is devoted to the proof of Proposition \ref{prop: more regularity for convolutions}, which states that the extra regularity established  in Theorem \ref{Thm A} for a solution $u$ of $u_t-F(D^2u)=0$ carries over to $u^+_\theta$ and $u^-_\theta$. 
\begin{prop}
\label{prop: more regularity for convolutions}
Assume  (F\ref{ellipticity}) and (F\ref{last prop of F}). Suppose $u\in C(\Omega\times(0,T))$ is a solution of $u_t-F(D^2u)=0$ in $\Omega \times(0,T)$. Assume $Y_r(\bar{x},\bar{t})\subset\Omega\times(0,T)$ and 
\begin{equation}
\label{asmregdist}
d(Y_r(\bar{x},\bar{t}), \bdry(\Omega\times(0,T))) \geq \theta.
\end{equation}
There exist universal constants $M_0$, $\sigma$, and $C$ such that for every $M\geq M_0$, there exists a set $\Psi^{+,\theta}_{M} \subset K^{\theta}_r ( \bar{x},\bar{t})$  (resp. $\Psi^{-,\theta}_{M} \subset K^{\theta}_r ( \bar{x},\bar{t})$) such that for any $(x, t) \in \Psi^{+,\theta}_M$ (resp. $(x, t) \in \Psi^{-,\theta}_M$), there exists a polynomial $P \in \PPi$ such that $P(x,t)=0$,
\begin{equation}
\label{cond on P}
P_t - F(D^2P)=0,
\end{equation}  
and if $Y_\rho^-(x,t)\subset Y^\theta_r(\bar{x},\bar{t})$, then for all  $(y,s)\in Y_\rho^-(x,t)$, we have
\begin{equation}
\label{eq: touch utheta from below}
\begin{split}
u^{+}_\theta(y,s)-u^{+}_\theta(x, t) \geq P(y,s)- nM\frac{\rho}{r}|x-y|^2-\frac{\rho}{r}\left(CM+\frac{ 1+\linfty{u}{\Omega \times (0,T)}}{\theta^2}\right)(t-s)
\end{split}
\end{equation}
\begin{equation*}
\begin{split}
\text{(resp. }u^{-}_\theta(y,s)-u^{-}_\theta(x, t) \leq P(y,s)+ nM\frac{\rho}{r}|x-y|^2 +\frac{\rho}{r}\left(CM+\frac{ 1+\linfty{u}{\Omega \times (0,T)}}{\theta^2}\right)(t-s).)
\end{split}
\end{equation*}
Moreover,
\begin{equation}
\label{eq: size of bad set theta}
|K^{\theta}_r(\bar{x},\bar{t}) \setminus \Psi^{\pm,\theta}_M |  \leq   \frac{Cr^{n+2}}{M^{\sigma}\theta^{n+\sigma}}(r^{-\sigma}+\theta^{-\sigma})\left(\linfty{u}{\Omega \times (0,T)}+1\right)^{n+\sigma}.
\end{equation}
\end{prop}

\subsubsection*{Outline of the proof of Proposition \ref{prop: more regularity for convolutions}.}
We prove that if  the supremum in the definition of $u^+_\theta(x,t)$ is achieved at a point  $(x^*,t)$ that   is contained in  the ``good set" $\Psi_M(u)$ of $u$, then an estimate similar to (\ref{eq: touch utheta from below}) holds at $(x,t)$. We state this as Lemma \ref{cor:psiMx*} below. This follows from Corollary \ref{cor} and the observation that  if $u$ is touched from below at $(x^*,t)$ by some function $\phi$, then $u^+_\theta$ is touched by below at $(x,t)$ by a translate of $\phi$.  We make this observation  precise in Lemma \ref{lem:touch}, which we  use in the proof of Lemma \ref{cor:psiMx*}. 

The previous discussion indicates that $\Psi^{+,\theta}_{M}$, the ``good set" for $u^+_\theta$, should be defined as the points $(x,t)$ for which  $(x^*,t)$ is in the good set $\Psi_M(u)$ of $u$. To obtain a bound on the size of $\Psi^{+,\theta}_{M}$  we  thus consider a  map that takes $x^*$ to $x$ and investigate what happens to  the size of  $\Psi_M(u)$ under this map. For this  we need Lemma \ref{lem:Dxu lip}, which says that $u$ is twice differentiable in $x$ at any point $(x^*,t)$ where the supremum or infimum is achieved in the regularizations $u^+_\theta$ or $u^-_\theta$, and gives a two-sided estimate for $D^2u$ at such points.

\subsubsection*{Statement of lemmas that we use in the proof of Proposition \ref{prop: more regularity for convolutions}}
\begin{lem}
\label{cor:psiMx*}
Let $Y_r(\bar{x},\bar{t})\subset \Omega\times(0,T)$ and let $\Psi(u,Y_r(\bar{x},\bar{t}))$ be the set given by Corollary \ref{cor}. 
Let us suppose that the supremum in the definition of $u^+_\theta(x,t)$ is achieved at $(x^*,t)\in \Psi_M(u,Y_r(\bar{x},\bar{t}))$. Then there exists a paraboloid $P\in \PPi$ that satisfies (\ref{cond on P}) and such that, for all $(y,s)\in Y^\theta_r(\bar{x},\bar{t})\cap \{s\leq t\}$,
\begin{equation}
\label{u+below}
\begin{split}
u^+_\theta (y,s)-u^+_\theta (x,t)\geq P(y,s) -\frac{nM}{r^2} &( r|x-y|^3 + |x-y|^2|t-s| +|t-s|^2) \\&-C\frac{M+\linfty{u_t}{Y_r(\bar{x},\bar{t})}}{r}|x-y|(t-s).
\end{split}
\end{equation}
\end{lem}

\begin{lem}
\label{lem:Dxu lip}
Under the assumptions of Proposition \ref{prop: more regularity for convolutions}, let $(x,t)\in Y^\theta_r(\bar{x},\bar{t})$. There exists a universal constant $C$ such that if $x^*$ is any point where the supremum (resp. infimum) is achieved in the definition of $u^+_{\theta}(x,t)$ (resp. $u^-_{\theta}(x,t)$), then $u(y,t)$ is twice differentiable in $y$ at $x^*$ with 
\[
-C(\linfty{u}{\Omega\times(0,T)}+1)\theta^{-2} I \leq D^2u(x^*, t)\leq \theta^{-1}I 
\]
\[ \left(\text{resp. } -\theta^{-1}I\leq D^2u(x^*, t)\leq C(\linfty{u}{\Omega\times(0,T)}+1)\theta^{-2}I.\right)
\]
\end{lem}

We postpone the proof of Lemmas \ref{cor:psiMx*} and \ref{lem:Dxu lip} and proceed with:
\begin{proof}[Proof of Proposition \ref{prop: more regularity for convolutions}]
We will give the proof for $u^+_\theta$; the arguments for $u^-_\theta$ are similar.
Fix $M$ and, to simplify the notation, denote the set $\Psi_M(u, Y_r(\bar{x},\bar{t}))$ given by Corollary \ref{cor} by just $\Psi_M$. We define the set $\mathcal{C}$ by
\begin{align*}
\mathcal{C}=&\left\{(x^*,t)\text{: $(x^*,t)$ is a point at which the supremum in the definition of $u^+_\theta(x,t)$ } \right.
\\&  \   \   \   \left. \text{is achieved, for some $(x,t)\in Y^\theta_r(\bar{x},\bar{t})$}\right\}.
\end{align*}
Item (\ref{item:Utheta}) of Proposition \ref{prop:x-sup-conv} implies $\mathcal{C}\subset Y_r(\bar{x},\bar{t})$.  
For any $(x^*,t)\in \mathcal{C}$, the map
\[
y \mapsto \left( u(y,t)-\frac{|x-y|^2}{2\theta}\right)
\]
has a local maximum at $x^*$, because the supremum in the definition of $u^+_\theta(x,t)$ is achieved at $(x^*,t)$.  Hence its derivative has a zero at $x^*$. Computing  the derivative and setting it equal to zero gives,
\[
Du(x^*, t)+\frac{(x-x^*)}{\theta}=0.
\] 
Rearranging this equality yields $x=x^*-\theta Du(x^*,t)$. We define the map $T:\Omega\times (0,T)\rightarrow \rr^n\times (0,T)$ by
\[
T(y,t) = (y-\theta Du(y,t),t),
\]
so that $(x,t)=T(x^*,t)$.
By Lemma \ref{lem:Dxu lip}, $Du$ is Lipschitz on $ \mathcal{C}$; hence, the map $T$ is too. Let us denote by $D_{y,t}T$ the matrix of derivatives of $T$ in both the space and time variables, and by  $D_yT$ the matrix of derivatives of  $T$ in only the space variable. 
We will now estimate the determinant of $D_{y,t}T$ on $\mathcal{C}$. We observe that $\det D_{y,t}T=\det D_yT$. Let us fix $(y,t)\in \mathcal{C}$.  We use the definition of the map $T$ to express $D_yT$ in terms of $D^2u$:
\[
D_yT(y,t) =
  I-\theta D^2u(y,t).
\]
We use Lemma \ref{lem:Dxu lip} to bound $-\theta D^2u(y,t)$ from above and below and obtain,
\[
0 \leq 
D_yT(y,t) \leq 
  I \cdot C\theta^{-1}( \linfty{u}{\Omega\times(0,T)}+1).
\]
The left-most inequality in the previous line says $D_yT(y,t)$ is a non-negative definite matrix. Hence, taking determinants preserves inequalities, and so we find 
\begin{equation}
\label{est:detT}
0\leq \det D_{y,t}T(y,t)\leq C\theta^{-n}(\linfty{u}{\Omega\times(0,T)}+1)^n,
\end{equation}
where we have also used $\det D_yT=\det D_{y,t}T$. 

We define the set $A$ by
\[
A=T( \mathcal{C}\cap (K_r(\bar{x},\bar{t})\setminus \Psi_M)).
\]
By the Area Formula (\cite[Chapter 3]{Evans Gariepy}), we have
\[
|A|\leq \int_{\mathcal{C}\cap (K_r(\bar{x},\bar{t})\setminus \Psi_M)} |\det D_{y,t}T(y,t)|\  dy \  dt .
\]
We use the bound (\ref{est:detT}) on $\det D_{y,t}T(y,t)$ and the bound (\ref{better bdn on Theta_M}) on $|K_r(\bar{x},\bar{t})\setminus \Psi_M|$ of Corollary \ref{cor} to estimate the right-hand side of the previous line from above and obtain,
\begin{equation}
\label{size A}
|A|\leq  \frac{Cr^{n+2}}{M^{\sigma}}(r^{-\sigma}\linfty{Du}{Y_r(\bar{x},\bar{t})}^\sigma+\linfty{u_t}{Y_r(\bar{x},\bar{t})}^\sigma)\frac{(\linfty{u}{\Omega\times(0,T)}+1)^n}{\theta^n}.
\end{equation}
We define the set $\Psi_M^{+,\theta}$ by
\[
\Psi_M^{+,\theta}:=K^\theta_r(\bar{x},\bar{t})\setminus A.
\]
We will now use Lemma \ref{cor:psiMx*} to prove that for each point $(x,t)$ of $\Psi_M^{+,\theta}$ there exists $P\in \PPi$ satisfying (\ref{cond on P}) and (\ref{eq: touch utheta from below}). To this end, let us fix $(x,t)\in \Psi_M^{+,\theta}$. Let $(x^*, t)$ be any point at which the supremum in the definition of $u^+_\theta(x,t)$ is achieved. We have $(x^*,t)\in \mathcal{C}$. 
Since $(x,t)\in K^\theta_r(\bar{x},\bar{t})$, by item (\ref{item:Utheta}) of Proposition \ref{prop:x-sup-conv} we also have $(x^*, t)\in K_r(\bar{x},\bar{t})$. In addition, since $(x,t)=T(x^*,t)$, and $(x,t)\notin A$ by the definition of $\Psi_M^{+,\theta}$, we find
\[
T(x^*,t) = (x,t) \notin A = T(\mathcal{C}\cap (K_r(\bar{x},\bar{t})\setminus \Psi_M)),
\]
where the second equality follows simply by recalling the definition of $A$. Therefore, $(x^*,t)\notin \mathcal{C}\cap (K_r(\bar{x},\bar{t})\setminus \Psi_M)$. 
Since we've already shown that $(x^*, t)$ is contained in  $\mathcal{C}\cap K_r(\bar{x},\bar{t})$,  this  implies  $(x^*,t) \in \Psi_M$. Thus, we have
\begin{equation*}
(x^*,t)\in K_r(\bar{x},\bar{t})\cap\Psi_M.
\end{equation*}
Since $(x^*,t)$ is a point   at which the supremum in the definition of $u^+_\theta(x,t)$ is achieved, this means we are in exactly the situation of Lemma \ref{cor:psiMx*}. Therefore, we apply the lemma and  conclude that there exists a paraboloid $P\in \PPi$ satisfying (\ref{cond on P}) and such that the inequality (\ref{u+below}) holds for all $(y,s)\in Y^\theta_r(\bar{x},\bar{t})\cap \{s\leq t\}$. 
Finally, if the backward cylinder $Y_\rho^-(x,t)$ is contained in $Y^\theta_r(\xbar, \tbar)$, then for any $(y,s)\in Y_\rho^-(x,t)$ we have $(y,s)\in Y^\theta_r(\bar{x},\bar{t})\cap \{s\leq t\}$ and so the inequality (\ref{u+below}) holds at $(y,s)$. Moreover, we have $|x-y|\leq \rho$ and $|t-s|\leq \rho^2$. Using this to bound the right-hand side of (\ref{u+below}) from below, we obtain, for all $(y,s)\in  Y_\rho^-(x,t)$,
\begin{equation*}
\begin{split}
u^+_\theta (y,s)-u^+_\theta (x,t)\geq P(y,s) -\frac{nM}{r^2} &( r\rho |x-y|^2  +\rho^2|t-s|) \\&-C\rho\frac{M+\linfty{u_t}{Y_r(\bar{x},\bar{t})}}{r}(t-s).
\end{split}
\end{equation*}
By assumption we have $\rho\leq r$; therefore, we have $\frac{\rho^2}{r^2}\leq \frac{\rho}{r}$. We use this to estimate the right-hand side of the previous line from below, and rearrange so that all the terms with $(t-s)$ are together, to obtain, for all $(y,s)\in  Y_\rho^-(x,t)$,
\begin{equation}
\label{touch below no bd}
\begin{split}
u^+_\theta (y,s)-u^+_\theta (x,t)\geq & P(y,s) -nM\frac{\rho}{r}|x-y|^2 -\frac{\rho}{r}\left(CM+\linfty{u_t}{Y_r(\bar{x},\bar{t})}\right)(t-s).
\end{split}
\end{equation}
Since, by assumption, $d(Y_r(\bar{x},\bar{t}), \bdry \Omega\times(0,T))\geq \theta$,  Proposition \ref{thm:Wang's c alpha thm} implies 
\[
\linfty{u_t}{Y_r(\bar{x},\bar{t})} \leq \theta^{-2}(\linfty{u}{\Omega\times(0,T)}+1)
\]
and
\[
\linfty{Du}{Y_r(\bar{x},\bar{t})} \leq \theta^{-1}(\linfty{u}{\Omega\times(0,T)}+1).
\]
We use these bounds  in (\ref{size A}) and (\ref{touch below no bd}) to complete the proof.
\end{proof}

We now give the proofs of the two lemmas, starting with  Lemma \ref{lem:Dxu lip}. First, we need to establish the following basic fact about how parabolic viscosity solutions relate to viscosity solutions in only $x$:

\begin{lem}
\label{lem: visc soln in t and x}
If $u$ is a Lipschitz in $t$ viscosity solution of $u_t-F(D^2u)=0$ in $Y_r$, then for any $t\in (0,r^2)$ the function $x\mapsto u(x,t)$ satisfies, in the viscosity sense,
\begin{align*}
\mathcal{M}^-(D^2u(x, t)) &\leq \linfty{u_t}{Y_r} \text{ in }B_r,\\
\mathcal{M}^+(D^2u(x, t)) &\geq -\linfty{u_t}{Y_r} \text{ in }B_r.
\end{align*}
\end{lem}
\begin{proof}
We will check that $u(\cdot, t)$ is a subsolution; checking that $u(\cdot, t)$ is a supersolution is analogous. Let us suppose  that $\phi(y)\in C^2(B_r)$ touches $u(y, t)$ from above at $x\in B_r$, so that
\begin{equation}
\label{uP}
u(y,t)\leq \phi(y) \text{ for all } y\in B_r, \text{ with equality at } y=x.
\end{equation}
Since $u$ is Lipschitz in $t$, we have
\[
u(y,s)\leq u(y,t)+ \linfty{u_t}{Y_r} (t-s)  \text{ for all $y\in B_r$ and for all $s\leq t$, with equality  at } s=t.
\]
Using the inequality (\ref{uP}) to bound the right-hand side of the previous line from above yields that $\phi(y)+ \linfty{u_t}{Y_r} (t-s)$ touches $u(y,s)$ from above at $(x,t)$ in  $B_r\times (0,t]$. Because $u$ is a viscosity solution of $u_t-F(D^2u)=0$, we obtain
\[
-\linfty{u_t}{Y_r} -F(D^2\phi(x)) \leq 0.
\]
The uniform ellipticity of $F$ implies
\[
0\geq -\linfty{u_t}{Y_r} -\mathcal{M}^+(D^2\phi(x)),
\]
as desired.
\end{proof}

We also recall the Harnack inequality for \emph{elliptic} equations  (see \cite[Theorem 4.3]{Cabre Caffarelli book}).
\begin{thm}
\label{thm:harnack}
Assume that $u:B_s(0)\subset \rr^n\rightarrow \rr$ satisfies $u\geq 0 $ in $B_s(0)$ and, for some positive constant $c$,
\begin{align*}
\mathcal{M}^-(D^2u) &\leq c \text{ in }B_s(0)\\
\mathcal{M}^+(D^2u) &\geq -c \text{ in }B_s(0).
\end{align*}
There exits a universal constant $C$ so that
\[
\sup_{B_{s/2}(0)} u \leq C\left(\inf_{B_{s/2}(0)} u + s^2c\right).
\]
\end{thm}

We are now ready to proceed with:
\begin{proof}[Proof of Lemma \ref{lem:Dxu lip}]
We will give the proof for $u^+_{\theta}(x,t)$; the argument for  $u^-_{\theta}(x,t)$ is very similar. 

Let $x^*$ be a point at which  the supremum in the definition of $u^+_\theta(x,t)$ is achieved, so that
\begin{equation}
\label{eq:diff at x*}
u(x^*,t) - \frac{|x-x^*|^2}{2\theta} \geq u(y,t) - \frac{|x-y|^2}{2\theta} \text{ for all } y\in \Omega. 
\end{equation}
Therefore, as a function of $y$, $u(y,t)$ is touched from above at $x^*$ by the convex paraboloid $Q(y)$, where
\[
Q(y)=u(x^*,t) - \frac{|x-x^*|^2}{2\theta} + \frac{|x-y|^2}{2\theta}.
\]
Thus, we have a bound from above on $D^2u(x^*,t)$:
\begin{equation*}
D^2u(x^*,t) \leq \theta^{-1}I.
\end{equation*}
Next we will show, using the Harnack inequality, that $u(\cdot, t)$ is also touched from below at $x^*$ by a paraboloid with bounded opening, thus obtaining a bound from below on $D^2u(x^*, t)$.

There exists an affine function $l(y)$ with $l(x^*)=u(x^*, t)$ and such that
\[
Q(y)= l(y)+\frac{|y-x^*|^2}{2\theta}.
\]
There exists $\rho>0$ such that $B_\rho(x^*)\subset B_r(\bar{x})$. We fix any $0<s<\rho$  and define the function $w(y)$ by 
\[
w(y)=l(y)+\frac{s^2}{2\theta} - u(y,t).
\]
We will apply the Harnack inequality, Theorem \ref{thm:harnack}, to $w$ in  $B_s(x^*)$ once we verify its hypotheses. First, we verify that $w$ is non-negative on $B_s(x^*)$. Let us fix $y\in B_s(x^*)$. We use the definition of $w$, the fact that $u\leq Q$ on $B_s(x^*)$ and the definition of $Q$, to obtain:
\[
w(y)=l(y)+\frac{s^2}{2\theta} - u(y,t)\geq l(y)+\frac{s^2}{2\theta} -Q(y)=  \frac{s^2}{2\theta}-\frac{|y-x^*|^2}{2\theta}  \geq 0.
\]
Next, by the properties of the extremal operators and Lemma \ref{lem: visc soln in t and x}, we see that $w$ satisfies
\begin{align*}
\mathcal{M}^-(D^2w) = \mathcal{M}^-(-D^2u)  = - \mathcal{M}^+(D^2u) &\leq \linfty{u_t}{Y_r(\bar{x},\bar{t})}\text{ in }B_s(x^*),\\
\mathcal{M}^+(D^2w)= \mathcal{M}^+(-D^2u) = -\mathcal{M}^-(D^2u)  &\geq - \linfty{u_t}{Y_r(\bar{x},\bar{t})} \text{ in }B_s(x^*).
\end{align*}
 Therefore, the Harnack inequality  implies that there exists a universal constant $C$ such that
\begin{equation}
\label{eqn:useharnack}
\sup_{B_{s/2}(x^*)}w \leq C\left(\inf_{B_{s/2}(x^*)}w + s^2\linfty{u_t}{Y_r(\bar{x},\bar{t})}\right).
\end{equation}
We evaluate $w$ at $x^*$ and use  that $l(x^*)=u(x^*,t)$ to obtain a upper bound on the infimum of $w$:
\[
\inf_{B_{s/2}(x^*)}w \leq w(x^*)=l(x^*)+\frac{s^2}{2\theta}-u(x^*,t)=\frac{s^2}{2\theta}.
\]
We use this to bound the first term on the left-hand side of (\ref{eqn:useharnack}) from above and thus obtain an upper bound on the supremum of $w$:
\[
\sup_{B_{s/2}(x^*)}w \leq Cs^2(\theta^{-1} +\linfty{u_t}{Y_r(\bar{x},\bar{t})}).
\]
Next we use this upper bound to estimate the difference between $Q$ and $u$ from above. We use the definitions of $Q$ and of $w$, and the  estimate on the supremum of $w$ in the previous line, to find, for all $y\in B_{s/2}(x^*)$, 
\[
Q(y)-u(y,t) = l(y)+\frac{|y-x^*|^2}{2\theta}-u(y,t) \leq w(y) \leq Cs^2(\theta^{-1} +\linfty{u_t}{Y_r(\bar{x},\bar{t})}).
\]
Since this holds for \emph{all} $0<s<\rho$, we have that $(Q(\cdot)-u(\cdot,t))$ is touched from above at $x^*$ by a convex paraboloid of opening $C(\theta^{-1} +\linfty{u_t}{Y_r(\bar{x},\bar{t})})$. Therefore, 
\[
D^2Q(x^*)-D^2u(x^*,t)\leq C(\theta^{-1} +\linfty{u_t}{Y_r(\bar{x},\bar{t})}))I.
\]
Rearranging the previous inequality and using  $D^2Q(x^*)=\theta^{-1}I$ yields
\begin{equation}
\label{D2ubelow}
D^2u(x^*, t) \geq -C(\theta^{-1} +\linfty{u_t}{Y_r(\bar{x},\bar{t})})I.
\end{equation}
Since assumption (\ref{asmregdist}) of this proposition says that $Y_r(\bar{x},\bar{t})$ is distance $\theta$ away from the boundary of $\Omega\times (0,T)$,  Proposition \ref{prop:diff of soln} implies
\[
\linfty{u_t}{Y_r(\bar{x},\bar{t})}\leq C\theta^{-2}(\linfty{u}{\Omega\times(0,T)}+1).
\] 
We use this to bound the right-hand side of (\ref{D2ubelow}) from below and obtain the desired estimate.
\end{proof}

Next, we give the proof of Lemma \ref{cor:psiMx*}. We need the following lemma, which we mentioned in the  outline of the proof of Proposition \ref{prop: more regularity for convolutions}.
\begin{lem}
\label{lem:touch}
Let us suppose that the supremum in the definition of $u^+_\theta(x,t)$ is achieved at $(x^*,t)\in Y_r(\bar{x},\bar{t})\subset \Omega\times(0,T)$. In addition, suppose that for some $\phi\in C(\Omega\times(0,T))$ we have 
\begin{equation}
\label{uphibelow}
u(z,s) - u(x^*,t)\geq \phi(z,s)
\end{equation}
for all $(z,s)\in Y_r(\bar{x},\bar{t})\cap\{s\leq t\}$ and  $\phi(x^*,t)=0$.  Then, for all $(y,s)\in Y^\theta_r(\bar{x},\bar{t})\cap \{s\leq t\}$,
\[
 u_\theta^+(y,s) - u_\theta^+(x,t) \geq \phi(y+x^*-x,s).
\]
\end{lem}
\begin{proof}[Proof of Lemma \ref{lem:touch}]
Since $(x^*,t)$ is a point at which the supremum is achieved in the definition of $u^+_\theta(x,t)$, we have
\begin{equation}
\label{atx*}
 u_\theta^+(x,t) =u(x^*,t)-\frac{|x-x^*|^2}{2\theta} .
\end{equation}
Let  $(y,s)$ be any point in  $Y^\theta_r(\bar{x},\bar{t})\cap \{s\leq t\}$.  By item (\ref{item:Utheta}) of Proposition  \ref{prop:x-sup-conv}, we have $(y+ x-x^*,s)\in Y_r(\bar{x},\bar{t})$. In particular, we have $(y+ x-x^*)\in \Omega$, so we may use $(y+ x^*-x)$  as a ``test point'' in  the definition of $u_\theta^+(y,s)$. We obtain
\[
 u_\theta^+(y,s) \geq u(y+x^*-x,s)-\frac{|x-x^*|^2}{2\theta}.
\]
Subtracting the equality (\ref{atx*}) from the above gives, for all $y\in Y^\theta_r(\bar{x},\bar{t})$,
\begin{equation}
\label{uthetu}
 u_\theta^+(y,s) - u_\theta^+(x,t) \geq u(y+x^*-x,s)-u(x^*,t).
\end{equation}
Since $(y+ x^*-x,s)\in Y_r(\bar{x},\bar{t})\cap \{s\leq t\}$,  we may evaluate (\ref{uphibelow}) at $z=y+ x^*-x$ to find,
\[
u(y+x^*-x,s)-u(x^*,t)\geq \phi(y+x^*-x,s).
\]
 We use this to bound the right-hand side of (\ref{uthetu}) and obtain,
\[
 u_\theta^+(y,s) - u_\theta^+(x,t) \geq \phi(y+x^*-x,s).
\]
This holds for any $(y,s)$ in $Y^\theta_r(\bar{x},\bar{t})\cap \{s\leq t\}$; hence the proof of the lemma is complete.
\end{proof}
We now proceed with:
\begin{proof}[Proof of Lemma \ref{cor:psiMx*}]
Since $(x^*,t) \in \Psi_M(u,Y_r(\bar{x},\bar{t}))$, Corollary \ref{cor} implies that there exists a polynomial $Q\in \PPi$ with $Q_t-F(D^2Q)=0$ such that for $(z,s)\in Y_r(\xbar,\tbar)\cap \{s\leq t\}$,
\begin{equation*}
\begin{split}
 |u(z,s) - u(x^*,t) - Q(z,s)| \leq \frac{nM}{r^2} &( r|x^*-z|^3 + |x^*-z|^2|t-s| +|t-s|^2) +\\&+C\frac{M+\linfty{u_t}{Y_r(\bar{x},\bar{t})}}{r}|x^*-z|(t-s).
 \end{split}
\end{equation*}
In particular, we have $u(z,s)-u(x^*,t) \geq \phi(z,s)$ for all $(z,s)\in Y_r(\xbar,\tbar)\cap \{s\leq t\}$, where we take $\phi$ to be
\begin{equation*}
\begin{split}
 \phi(z,s)= Q(z,s)- \frac{nM}{r^2} &( r|x^*-z|^3 + |x^*-z|^2|t-s| +|t-s|^2) -\\&-C\frac{M+\linfty{u_t}{Y_r(\bar{x},\bar{t})}}{r}|x^*-z|(t-s).
 \end{split}
\end{equation*}
We note $\phi(x^*,t)=Q(x^*,t)=0$. Therefore, according to Lemma \ref{lem:touch}, we have,
\begin{equation}
\label{withphi}
 u_\theta^+(y,s) - u_\theta^+(x,t) \geq \phi(y+x^*-x,s) \text{ for all }(y,s)\in Y^\theta_r(\bar{x},\bar{t})\cap \{s\leq t\}.
\end{equation}
Let us now use the definition of $\phi$ to obtain an alternate expression for the right-hand side of (\ref{withphi}):
\begin{align*}
\phi(y+x^*-x,s)=Q(y+x^*-x,s) - &\frac{nM}{r^2} ( r|x-y|^3 + |x-y|^2|t-s| +|t-s|^2) -\\&-C\frac{M+||u||_{L^{\infty}(Y_r(\bar{x},\bar{t}))}}{r}|x-y|(t-s).
\end{align*}
Let us define the paraboloid  $P(y,s)\in \PPi$ by $P(y,s)= Q(y-x+x^*,s)$. Using this definition and the above expression for $\phi(y+x^*-x,s)$ in the inequality (\ref{withphi}) gives that (\ref{u+below}) holds for all $(y,s)\in Y^\theta_r(\bar{x},\bar{t})\cap \{s\leq t\}$. Finally, we observe that $P$ satisfies the equation (\ref{cond on P}): indeed, $P$ is only a translate of $Q$, and since $Q\in \PPi$, we have that $Q_t(y,s)$ and $D^2Q(y,s)$ do not depend on $(y,s)$. Therefore, we have $P_t-F(D^2P)=Q_t-F(D^2Q)=0$.
\end{proof}

\section{A key estimate between solutions and sufficiently regular $\delta$-solutions}
\label{section: main prop}
In this section we state and prove Proposition \ref{main lemma}, a key part of the proofs of the two main results, Theorems \ref{thm:vague comparison} and \ref{thm:mesh comparison vague}. Roughly, Proposition \ref{main lemma} says that if $w$, $u$, and $v$ are, respectively, a $\delta$-subsolution, a solution and a $\delta$-supersolution of (\ref{eqn:mainLemma}) that satisfy
\[
w\leq u\leq v 
\]
on the parabolic boundary of a region, and $v$ and $w$ are sufficiently regular, then we have 
\[
w-\tilde{c}\delta^{\tilde{\alpha}}\leq u\leq v +\tilde{c}\delta^{\tilde{\alpha}}
\]
on the interior of the region.
Once we have proven Proposition \ref{main lemma}, we will only need to use the basic properties of inf- and sup- convolutions to verify its hypotheses and establish the two main results. We define the  constant $\zeta$ by
\begin{equation}
\label{constants}
 \zeta = \frac{\sigma}{2(4n+3+3\sigma)},
\end{equation}
where $\sigma$ is the constant from Proposition \ref{prop: more regularity for convolutions}. We remark that since $\sigma$ is  universal, so is $\zeta$.
\begin{prop}
\label{main lemma}
Assume (F1), (F2). Assume that $\Omega'$ satisfies (U1) and let $(a,b)$ be a subset of $(0,T)$ for some $T>0$. We denote 
\[
U=\Omega'\times(a,b).
\]
 Assume $u\in C^{0,1}_x(U)$ is a solution of 
\begin{equation}
\label{eqn:mainLemma}
u_t-F(D^2u)=0 
\end{equation}
in $U$. For $\delta\in (0,1)$, define the quantity $r_\delta$ by
\[
r_\delta = 9\sqrt{n}(1+\linfty{u}{U}+2\linfty{Du}{U})\delta^{\zeta}
\]
and the set $\tilde{U}$ by
\[
\tilde{U} = \left\{(x,t)\in U:\   \   t\leq b-r_\delta^2 \text{ and } d((x,t),\bdry U) > 2r_\delta \right\}.
\]
There exist universal constants $\tilde{\alpha}$ and $\tilde{\delta}$ and a positive constant $\tilde{c}$ that depends on $n$, $\lambda$, $\Lambda$, $\diam\Omega'$, $T$, $\linfty{u}{\U}$ and $\linfty{Du}{\U}$ such that for all $\delta\leq \tilde{\delta}$ and:
\begin{enumerate}
\item for  $v\in C^{0,1}(U)$ a $\delta$-supersolution of (\ref{eqn:mainLemma}) in $U$ that satisfies
\begin{enumerate}
\item \label{assump:d2v} $D^2v(x,t)\leq \delta^{-\zeta}I$ in the sense of distributions  for every $(x,t)\in U$, 
\item \label{assump:vt} $\tlipsemi{v}{\U}\leq 3T\delta^{-\zeta}$, and
\item \label{main prop wlog}
$\sup_{\bdry \tilde{U}}( u-v) \leq 0$,
\end{enumerate}
we have
\[
\sup_{\tilde{U}} (u-v)\leq \tilde{c} \delta^{\tilde{\alpha}} ; 
\]
\item for  $w\in C^{0,1}(U)$  a $\delta$-subsolution of (\ref{eqn:mainLemma}) in $U$ with $D^2w(x,t)\geq -\delta^{-\zeta}I$ for every $(x,t)\in U$ in the sense of distributions,   $\tlipsemi{w}{\U}\leq 3T\delta^{-\zeta}$ and  $\sup_{\bdry \tilde{U}}( w-u) \leq 0$, we have
\[
\sup_{\tilde{U}}( w-u)\leq \tilde{c} \delta^{\tilde{\alpha}}.
\]
\end{enumerate}
\end{prop}

\subsection{Outline}
\label{subsec:outline}
We outline the proof of part (1) of the Proposition; the proof of part (2) is very similar. The main idea of the proof is to control the supremum of $(u-v)$ by the size of the contact set of $(u-v)$ with the upper monotone envelope of $(u-v)$.

We first regularize $u$  by $x$-sup convolution and obtain $u^+_{\delta^{\zeta}}$ (see Definition \ref{defxinf}). We then perturb $u^+_{\delta^{\zeta}}$ by subtracting $\delta^{1/4}t$ to obtain a strict sub-solution.  We will  bound $\sup (u^+_{\delta^{\zeta}}-\delta^{1/4}t -v)$. Since $u^+_{\delta^{\zeta}}$ and $v$ are sufficiently regular, we are able to apply Proposition \ref{pre ABP} and find 
\[
\sup_{\tilde{U}} (u^+_{\delta^{\zeta}}-\delta^{1/4}t -c_1\delta^\zeta-v) \leq C|\{u^+_{\delta^{\zeta}}-\delta^{1/4}t-c_1\delta^\zeta-v=\bar{\Gamma}\}|^{\frac{1}{n+1}}\delta^{-C\zeta},
\]
where $\bar{\Gamma}$ is the upper monotone envelope of $u^+_{\delta^{\zeta}}-\delta^{1/4}t-c_1\delta^\zeta-v$ and $C$ and $c_1$ are positive constants. This is Lemma \ref{lem:helper1.5} below. We will proceed by contradiction and assume that $\sup_{\tilde{U}} (u^+_{\delta^{\zeta}}-\delta^{1/4}t-c_1\delta^\zeta -v)$ is ``large". By the estimate above, this implies that the size of the contact set $\{u^+_{\delta^{\zeta}}-\delta^{1/4}t-c_1\delta^\zeta-v = \bar{\Gamma}\}$ is large as well. 

The key part of our argument is Proposition \ref{prop: more regularity for convolutions}, which states that there exists a large subset $\Psi$ of $U$ on which $u^+_{\delta^{\zeta}}$ is very close to being a polynomial. We use this proposition, together with the fact that the size of the contact set is large, to find a point in the intersection of $\Psi$ and the contact set. This is Lemma \ref{lem:Helper2} below. It follows from Proposition \ref{prop: more regularity for convolutions} by a covering argument.

Our last ingredient is the fact that if $(x,t)$ is a point in the contact set $\{u^+_{\delta^{\zeta}}-\delta^{1/4}t-c_1\delta^\zeta-v = \bar{\Gamma}\}$ at which $u^+_{\delta^{\zeta}}$ is touched from below by some $\phi$, then $v$ is touched from below by $\phi$ at that point as well. This is Lemma \ref{lem:helper3} below. We use this fact, applied at the point  $(x,t)$ in the intersection of $\Psi$ and the contact set, to obtain the desired contradiction.

\subsection{Proof of Proposition \ref{main lemma}}
Throughout the rest of this section, we will use  $C$ and $C_i$ with $i=1,2,...$,  to denote positive constants that depend only $n$, $\lambda$, $\Lambda$, $\diam\Omega'$, $T$, $\linfty{u}{\U}$ and $\linfty{Du}{\U}$.
For the proof of Proposition \ref{main lemma} we need three  lemmas.

\begin{lem}
\label{lem:helper1.5}
Under the assumptions of  Proposition \ref{main lemma}, let $\bar{x}\in \rr^n$ and $\rho>0$ be such that $\tilde{U}\subset Y_{\rho}^-(\bar{x}, b-r^2_\delta)$. There exists a  constant $c_1>0$ that depends only on $\linfty{Du}{\U}$  and a constant $C_1$ so that
\begin{equation}
\label{eq:lemsup}
\sup_{\bdry \tilde{U}}(u_{\delta^\zeta}^{+} -\delta^{1/4}t-v)\leq c_1\delta^{\zeta},
\end{equation}
and
\begin{equation}
\label{bd cont set below}
|\{u_{\delta^\zeta}^{+} -\delta^{1/4}t-c_1\delta^{\zeta} - v =\bar{\Gamma}\}\cap \tilde{U}|\geq C_1 \delta^{2\zeta(n+1)}\left(\sup_{\tilde{U}}(u_{\delta^\zeta}^{+} -\delta^{1/4}t-c_1\delta^{\zeta} - v) \right)^{(n+1)},
\end{equation}
where $\bar{\Gamma}$ is the upper monotone envelope of $(u_{\delta^\zeta}^{+} -\delta^{1/4}t-c_1\delta^{\zeta}-v)^+$ extended by $0$ to $Y_{2\rho}^-(\bar{x}, b-r^2_\delta)$.
\end{lem}

\begin{lem}
\label{lem:Helper2}
Under the assumptions of Proposition \ref{main lemma}, there  exists a positive constant $\delta_1$ that depends only on $n$, a universal constant $\bar{C}$, and a positive constant $\bar{c}$ that depends on $n$, $\lambda$, $\Lambda$, $\diam\Omega'$, $\linfty{u}{U}$ and $\linfty{Du}{\U}$, such that, if $\delta\leq \delta_1$ and $\mathcal{C}\subset \tilde{U}$ with 
\[
|\mathcal{C}|\geq \bar{c} \delta^{3\zeta (n+1)},
\]
there then there exists $(x_0,t_0)\in \mathcal{C}$  and a paraboloid $P\in \PPi$ with
\begin{enumerate}
\item $Y^-_\delta(x_0,t_0)\subset U$,
\item $P(x_0,t_0)=0$,
\item $\displaystyle P_t-F(D^2P)=0$, and 
\item for all $(x, t)\in Y_{\delta}^-(x_0,t_0)$,
\begin{equation}
\label{touch u at x0t0}
\begin{split}
u_{\delta^\zeta}^{+}(x,t)-u_{\delta^\zeta}^{+}(x_0, t_0) \geq P(x,t) - \bar{C}\delta^{1/2}(|x-x_0|^2+(t_0-t)).
\end{split}
\end{equation}
\end{enumerate}
\end{lem}

Lemma \ref{lem:Helper2} follows from Proposition \ref{prop: more regularity for convolutions} and a covering argument. The covering argument is slightly more involved than in the elliptic case because of the need to avoid times $t$ that are close to the terminal time $T$.

\begin{lem}
\label{lem:helper3}
Under the assumptions of Proposition \ref{main lemma}, suppose there exists a point $(x_0,t_0)$ and a function $\phi(x,t)$ such that
\[
(x_0,t_0)\in \{u_{\delta^\zeta}^{+} -\delta^{1/4}t-c_1\delta^{\zeta} - v =\bar{\Gamma}\}\cap \tilde{U},
\]
 where $\bar{\Gamma}$ is as in Lemma \ref{lem:helper1.5}, and 
 \begin{enumerate}
\item $Y^-_\delta(x_0,t_0)\subset U$,
\item $\phi(x_0,t_0)=0$,
\item for all $(x, t)\in Y_{\delta}^-(x_0,t_0)$,
\begin{equation}
\label{touch u at x0t0phi}
\begin{split}
 u_{\delta^\zeta}^{+}(x,t)-u_{\delta^\zeta}^{+}(x_0, t_0) \geq \phi(x,t).
\end{split}
\end{equation}
\end{enumerate}
Then 
\begin{equation}
\label{conclem3}
\phi_t(x_0, t_0)-F(D^2\phi(x_0, t_0))\geq \delta^{1/4}.
\end{equation}
\end{lem}
In the situation of Lemma \ref{lem:helper3}, we have  that $u_{\delta^\zeta}^{+}$ is a viscosity \emph{subsolution}, so we cannot directly deduce (\ref{conclem3}) from (\ref{touch u at x0t0phi}). The proof of the lemma uses that $(x_0,t_0)$ is in the contact set and the fact that $v$ is a $\delta$-supersolution.

We postpone the proofs of the lemmas and proceed with the proof of Proposition \ref{main lemma}.
\begin{proof}[Proof of Proposition \ref{main lemma}]
We will give the proof for part (1) of the proposition; the proof of part (2) is very similar.  Let $c_1$ and $C_1$ be the constants from Lemma \ref{lem:helper1.5} and let $\delta_1$, $\bar{c}$ and $\bar{C}$ be the constants from Lemma \ref{lem:Helper2}. 
Let us take 
\begin{equation}
\label{choicedelta}
\tilde{\delta}=\min\left\{\delta_1, \frac{1}{2}\left(\bar{C}\left(1+\lambda/2\right)\right)^{-2}\right\}.
\end{equation}
We point out that  $\tilde{\delta}$ is universal, because $\delta_1$ depends only on $n$ and $\bar{C}$ is universal. Let us fix $\delta\leq \tilde{\delta}$. We claim
\begin{equation}
\label{main lemma bound sup}
\sup_{\tilde{U}}( u_{\delta^\zeta}^{+} -\delta^{1/4}t-c_1\delta^{\zeta} - v )\leq \left(\frac{\bar{c}}{C_1}\right)^{\frac{1}{n+1}}\delta^{\zeta}.
\end{equation}
If (\ref{main lemma bound sup}) holds, then we obtain the desired bound on $\sup(u-v)$: indeed, since $u\leq u^+_{\delta^{\zeta}}$ and $t\leq b\leq T$, we have,
\[
\sup_{\tilde{U}} (u-v) \leq \sup_{\tilde{U}} (u_{\delta^\zeta}^{+} -\delta^{1/4}t-c_1\delta^{\zeta}-v )+ \delta^{1/4}T +c_1\delta^{\zeta}\leq \left(\frac{\bar{c}}{C_1}\right)^{\frac{1}{n+1}}\delta^{\zeta}  + \delta^{1/4}T +c_1\delta^{\zeta} \leq \tilde{c}\delta^{\tilde{\alpha}},
\]
where $\tilde{c}= c_1+T+ \left(\frac{\bar{c}}{C_1}\right)^{\frac{1}{n+1}}$ and $\tilde{\alpha}=\min\{1/4, \zeta\}$. Therefore, to complete the proof of this proposition it will suffice to show that (\ref{main lemma bound sup}) holds. To this end,  we proceed by contradiction and assume 
\begin{equation}
\label{contsup}
\sup_{\tilde{U}}( u_{\delta^\zeta}^{+} -\delta^{1/4}t-c_1\delta^{\zeta} - v )> \left(\frac{\bar{c}}{C_1}\right)^{\frac{1}{n+1}}\delta^{\zeta}.
\end{equation}
 By Lemma \ref{lem:helper1.5}, we have the following lower bound on the size of the contact set:
\begin{equation*}
|\{u_{\delta^\zeta}^{+} -\delta^{1/4}t-c_1\delta^{\zeta} - v =\bar{\Gamma}\}\cap \tilde{U}|\geq C_1 \delta^{2\zeta(n+1)}\left(\sup_{\tilde{U}}(u_{\delta^\zeta}^{+} -\delta^{1/4}t-c_1\delta^{\zeta} - v) \right)^{(n+1)}.
\end{equation*}
We use (\ref{contsup}) to bound the right-hand side of the previous line from below and obtain,
\[
|\{u_{\delta^\zeta}^{+} -\delta^{1/4}t-c_1\delta^{\zeta} - v =\bar{\Gamma}\}\cap \tilde{U}|\geq \bar{c}\delta^{3\zeta(n+1)}.
\]
We now apply Lemma \ref{lem:Helper2} with $\mathcal{C}=\{u_{\delta^\zeta}^{+} -\delta^{1/4}t-c_1\delta^{\zeta} - v =\bar{\Gamma}\}\cap \tilde{U}$ and obtain that there exists a point $(x_0, t_0)\in \{u_{\delta^\zeta}^{+} -\delta^{1/4}t-c_1\delta^{\zeta} - v =\bar{\Gamma}\}$ with
\begin{equation}
\label{Y^-in U}
Y^-_\delta(x_0,t_0)\subset \tilde{U},
\end{equation}
and a paraboloid $P\in \PPi$ with $P(x_0,t_0)=0$,
\begin{equation}
\label{what P solves}
P_t-F(D^2P)=0,
\end{equation}
and
\begin{equation}
\label{touch u at x0t0 used}
\begin{split}
u_{\delta^{\zeta}}^+(x,t)-u_{\delta^{\zeta}}^+(x_0, t_0) \geq P(y,s) - \bar{C}\delta^{1/2}\cdot (|x_0-x|^2+(t_0-t)),
\end{split}
\end{equation}
where $\bar{C}$ is a universal constant. 
We see that the hypotheses of Lemma \ref{lem:helper3} are satisfied, with $\phi(y,s)=P(y,s) - \bar{C}\delta^{1/2}\cdot (|x_0-x|^2+(t_0-t))$.
Therefore, applying Lemma \ref{lem:helper3} yields that inequality (\ref{conclem3}) holds. With our choice of $\phi$,  (\ref{conclem3})  reads,
\[
P_t+\bar{C}\delta^{1/2}  - F\left(D^2P - \frac{\bar{C}\delta^{1/2}}{2}I\right)\geq \delta^{1/4}
\]
(we remark that since $P\in \PPi$, we have that $P_t$ and $D^2P$ are constant.)
We use the uniform ellipticity of $F$ to obtain an upper bound for the left-hand side of the previous line, and find 
\[
P_t+\bar{C}\delta^{1/2} - F(D^2P) +\lambda\frac{\bar{C}\delta^{1/2}}{2}\geq \delta^{1/4}.
\]
But,  according to (\ref{what P solves}), we have $P_t-F(D^2P)=0$, so we use this to simplify the left-hand side of the previous line  and obtain,
\[
\bar{C}\delta^{1/2}\left(1+\lambda/2\right)\geq \delta^{1/4}.
\]
Multiplying both sides by $\delta^{-1/2}$ we find,
\[
\bar{C}\left(1+\lambda/2\right)\geq \delta^{-1/2}.
\]
Since we took $\delta\leq \tilde{\delta}$, this inequality contradicts our choice of $\tilde{\delta}$ in (\ref{choicedelta}). Therefore, 
we have obtained the desired  contradiction to (\ref{contsup}), so (\ref{main lemma bound sup}) must hold. The proof of the proposition is thus complete.
\end{proof}

We now proceed with the proofs of Lemmas \ref{lem:helper1.5}, \ref{lem:Helper2} and \ref{lem:helper3}.

\begin{proof}[Proof of Lemma \ref{lem:helper1.5}]
We first establish the estimate (\ref{eq:lemsup}). To this end, by item (\ref{item: bound from above on sup conv}) of  Proposition \ref{prop:x-sup-conv} and assumption (\ref{main prop wlog}) of Proposition \ref{main lemma} we have
\begin{equation*}
 \sup_{\bdry \tilde{U}} (u_{\delta^\zeta}^{+} -\delta^{1/4}t- v)\leq \sup_{\bdry \tilde{U}} (u_{\delta^\zeta}^{+} -v )\leq  \sup_{\bdry \tilde{U}} (u-v) +2\delta^{\zeta}\linfty{Du}{U}^2\leq c_1\delta^{\zeta},
\end{equation*}
where $c_1=2\linfty{Du}{\U}^2$. Let us define the function $w$ on $Y^-_\rho(\bar{x}, b-r_\delta^2)$ by 
\[
w(x,t)=\begin{cases}
\max\{u_{\delta^\zeta}^{+}(x,t) -\delta^{1/4}t-v(x,t)-c_1\delta^\zeta, 0\} &\text{ for }(x,t)\in\tilde{U}\\
0& \text{ on }Y^-_\rho(\bar{x}, b-r_\delta^2)\setminus \tilde{U}.
\end{cases}
\]
The previous estimate implies that $w$ is continuous on $Y^-_\rho(\bar{x}, b-r_\delta^2)$.

To establish the second assertion of the lemma, estimate (\ref{bd cont set below}),  we seek to apply Proposition \ref{pre ABP} to $-w$.
Since $d(\tilde{U}, \bdry U) =2r_\delta$, Proposition \ref{prop:diff of soln} implies that $u$ is differentiable in $t$ in $\tilde{U}$, with 
\[
\linfty{u_t}{\tilde{U}}\leq C r^{-2}_\delta (\linfty{u}{\U}+1)\leq C\delta^{-2\zeta}.
\]
Therefore, $\tlipsemi{u_{\delta^\zeta}^{+}}{\tilde{U}}\leq C\delta^{-2\zeta}$ as well. Together with assumption (\ref{assump:vt}) of Proposition \ref{main lemma}, this implies $\tlipsemi{w}{\tilde{U}}\leq C\delta^{-2\zeta}$. In addition, assumption (\ref{assump:d2v}) of Proposition \ref{main lemma} and the properties of $x$-sup convolutions  (item (\ref{item:sup conv is semiconvex}) of Proposition \ref{prop:x-sup-conv})) imply $-D^2w=D^2(-u_{\delta^\zeta}^{+} + v)\leq 2\delta^{-\zeta}I$ in the sense of distributions on all of $\tilde{U}$. 
Since $w\equiv 0$ outside of $\tilde{U}$, the hypotheses of  Proposition \ref{pre ABP} hold with $K=C\delta^{-2\zeta}$. Applying  Proposition \ref{pre ABP} gives a bound on the supremum of $(-w)^-$ in terms of the size of the contact set of $-w$ with $\Gamma$, the lower monotone envelope of $\min(-w,0)$:
\[
\sup_{Y_{2\rho}^-(\bar{x}, b-r_\delta^2)} (-w)^-\leq C|\{-w=\Gamma\}|^{\frac{1}{n+1}} \delta^{-2\zeta}. 
\]
It is easy to see  $\bar{\Gamma}\equiv -\Gamma$, where $\bar{\Gamma}$ is defined in the statement of this lemma, and $\sup_{Y_{2\rho}^-(\bar{x}, b-r_\delta^2)}(-w)^-=\sup_{\tilde{U}} w$. Using this, together with the previous estimate and the definition of $w$ yields,
\[
\sup_{\tilde{U}}(u_{\delta^\zeta}^{+} -\delta^{1/4}t-c_1\delta^{\zeta} - v ) \leq C|\{u_{\delta^\zeta}^{+} -\delta^{1/4}t-c_1\delta^{\zeta} - v =\bar{\Gamma}\}\cap \tilde{U}|^{\frac{1}{n+1}} \delta^{-2\zeta}.
\]
Rearranging, we find that (\ref{bd cont set below}) holds.
\end{proof}

\begin{proof}[Proof of Lemma \ref{lem:Helper2}]
Let us fix a subset $\mathcal{C}$ of  $\tilde{U}$ that satisfies
\begin{equation}
\label{sizeC}
|\mathcal{C}|\geq \bar{c}\delta^{3\zeta(n+1)},
\end{equation}
where the constant $\bar{c}$ is specified in line (\ref{choicec}) below, and depends only on $n$, $\lambda$, $\Lambda$, $\diam\Omega'$, $T$, $\linfty{u}{\U}$ and $\linfty{Du}{\U}$. To simplify notation, for the remainder of this proof we denote the $x$-sup convolution $u^+_{\delta^{\zeta}}$ by simply $u^+$. In addition, we  need the following families of sets:
\begin{align*}
\ydel(x,t) &:= Y^{\delta^{\zeta}}_{\rdel} (x,t) = B\left(x, r_\delta -2 \delta^{\zeta} \linfty{Du}{\U}\right) \times \left(t , t+r_\delta^2\right],\\
\kdel(x,t)  &:= K^{\delta^{\zeta}}_{\rdel} (x,t) =\left[ x - \delta^{\zeta}(1+\linfty{u}{U}) ,x+ \delta^{\zeta}(1+\linfty{u}{U})\right]^n \times \left(t , t+\frac{r_\delta^2}{81n}\right], \\
\kdeltop(x,t)  &:=\left[ x- \delta^{\zeta}(1+\linfty{u}{U}) ,x + \delta^{\zeta}(1+\linfty{u}{U})\right]^n \times \left(t+\frac{r_\delta^2}{162n} , t+\frac{r_\delta^2}{81n}\right].
\end{align*}
We point out that  $\kdeltop(x,t)$ is the top half (in terms of $t$) of $\kdel(x,t)$. See Figure 1.
\begin{figure}\caption{The sets involved in the proof of Lemma \ref{lem:Helper2}.}  \label{sets for lemma}\centering \begin{tikzpicture}[scale=.8]
\draw (0,0) rectangle (10,6);
\draw (2,0) rectangle (8,2);
\draw (2,2) rectangle (8,4);
\draw [fill=lightgray] (1,1) rectangle (7,2.5);
\draw [fill] (5,0) circle (.1);
\node [below] at (5,0) {$(x,t)$};
\draw [fill] (4,2.5) circle (.1);
\node [above] at (4,2.5) {$(x_0,t_0)$};
\node [above left] at (7,1) {$Y^-_\delta(x_0,t_0)$};
\node [below right] at (0,6) {$\tilde{Y}_{r_\delta}(x,t)$};
\node [below left] at (8,4) {$\tilde{K}^T_{r_\delta}(x,t)$};
\end{tikzpicture}\end{figure}
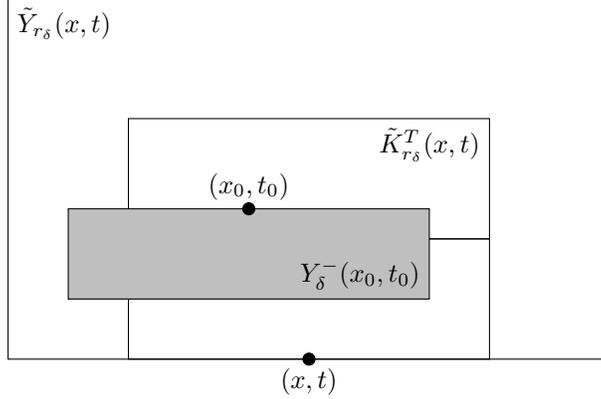
We remark that there exists a constant $\delta_1$ that depends only on $n$ such that if $\delta\leq \delta_1$, then  
\begin{equation}
\label{remark about sets}
\text{if $(x,t)\in \tilde{U}$ and $(x_0,t_0)\in \kdeltop(x,t)$, then $Y_{\delta}^-( x_0,t_0) \subset \ydel(x,t)\cap U $.}
\end{equation}
(We point out that in order for (\ref{remark about sets}) to hold,  it is important that $(x_0,t_0)$ is contained in  $\kdeltop(x,t)$, and not only in $\kdel(x,t)$.)  We take $\delta\leq\delta_1$.

We will be applying Proposition \ref{prop: more regularity for convolutions}. We first remark that although it is stated in $\Omega\times (0,T)$, it holds in $U=\Omega'\times (a,b)$ as well, with no other modifications. 

We recall that Proposition \ref{prop: more regularity for convolutions} gives a lower bound on the size of $\Psi_M^{+,\delta^{2\zeta}}(u^+)$, the set where $u^+$ has second-order expansions from below.  For the rest of the argument we will denote  $\Psi_M^{+,\delta^{2\zeta}}(u^+)$ simply by $\Psi_M^+$. We will show that for $M$ large enough there exists a point in the intersection of $\Psi_M^+$ and $\mathcal{C}$ by using the lower bounds on the sizes of $\mathcal{C}$ and  $\Psi_M^{+}$ (the latter is proved in Proposition \ref{prop: more regularity for convolutions}). Since Proposition  \ref{prop: more regularity for convolutions} bounds the size of the complement of $\Psi_M^{+}$ inside of sets of the form $\kdel(x,t)$, we first cover $\tilde{U}$ by $\{ \kdeltop(x,t): \  \  (x,t)\in \tilde{U}\}$. It is clear that  there exits a  finite collection $\{ \kdeltop(x_j,t_j): \  \  (x_j,t_j)\in \tilde{U}\}_{j=1}^s$ that covers $\tilde{U}$, where $s \leq  C_0 \delta^{-\zeta(n+2)}$, for $C_0$ a constant  that depends on $\diam \Omega '$, $T$, and $n$ . Since $\mathcal{C}\subset \tilde{U}$ and the $ \kdeltop(x_j,t_j)$ cover $\tilde{U}$, there exists an $i$ such that 
\[
|\mathcal{C}\cap \kdeltop(x_i,t_i)| \geq \frac{|\mathcal{C}|}{s} .
\]
We use (\ref{sizeC}) and the upper bound on $s$ to estimate the right-hand side of the previous line from below. We find,
\[
|\mathcal{C}\cap \kdeltop(x_i,t_i)| \geq \frac{\bar{c} \delta^{3\zeta (n+1)}}{C_0 \delta^{-\zeta (n+2)}} = \frac{\bar{c}}{C_0} \delta^{\zeta (4n+5) }.
\]
In view of the definition of $\tilde{U}$, we have $d((x_i,t_i), \bdry U)\geq 2r_\delta$ and  $t_i\leq b-r_\delta^2$; therefore, $Y_{r_\delta}(x_i,t_i)\subset \U$ and 
\begin{equation*}
d(Y_{r_\delta}(x_i,t_i), \bdry U)\geq r_\delta > \delta^{\zeta}.
\end{equation*}
Therefore, $u$  satisfies the hypotheses of Proposition \ref{prop: more regularity for convolutions} in $Y_{r_\delta}(x_i,t_i)$ with $\theta = \delta^{\zeta}$, $\rho=\delta$ and $r=r_\delta$. Let $M_0$ be the universal constant from Proposition \ref{prop: more regularity for convolutions}. Thus, for any $M>M_0$, there exists a set $\Psi^{+}_M$ (the ``good set" of $u^+$) with
\begin{equation}
\label{sizeKtil}
|\kdel(x_i,t_i)\setminus\Psi^{+}_M|\leq \frac{C_1r_\delta^{n+2}}{M^{\sigma}\delta^{\zeta(n+\sigma)}}(r_\delta^{-\sigma}+\delta^{-\zeta\sigma})\left(\linfty{u}{U}+1\right)^{n+\sigma},
\end{equation}
where $C_1$ is a universal constant. We take
\[
M=\left(\frac{2C_1 C_0 r_\delta^{n+2}\delta^{-\zeta(n+\sigma)}(r_\delta^{-\sigma}+\delta^{-\zeta \sigma})(\linfty{u}{\U}+1)^{n+\sigma}}{\bar{c}\delta^{\zeta (4n+5) }}\right)^{1/\sigma} +M_0 .
\]
We use our choice of $M$ to bound the right-hand side of (\ref{sizeKtil}) from above, and obtain
\[
|\kdel(x_i,t_i)\setminus\Psi^{+}_M|< \frac{\bar{c}}{C_0} \delta^{\zeta(n+\sigma)}=|\mathcal{C}\cap \kdeltop(x_i,t_i)| .
\]
Since we have $\kdeltop(x_i,t_i)\subset \kdel(x_i,t_i)$, the previous inequality implies,
\[
|\kdeltop(x_i,t_i)\setminus\Psi^{+}_M|<|\kdel(x_i,t_i)\setminus\Psi^{+}_M|< |\mathcal{C}\cap \kdeltop(x_i,t_i)| .
\]
Therefore, there exists a point $(x_0,t_0)\in \mathcal{C}\cap \Psi^{+}_M\cap \kdeltop(x_i,t_i)$. Because of (\ref{remark about sets}), we have $Y_{\delta}^-( x_0,t_0) \subset \ydel(x,t)\cap U $. Thus, by Proposition \ref{prop: more regularity for convolutions}, there exists a polynomial $P\in \PPi$ such that 
\[
P_t-F(D^2P)=0
\]
and, on  $Y_{\delta}^-( x_0,t_0)$,
\begin{equation}
\label{touch u+ belowone}
u^{+}(x,t)-u^{+}(x_0, t_0) \geq P(y,s)-C_2M\frac{\delta}{r_\delta}|x-x_0|^2 - \left(\frac{\delta}{r_\delta}C_2M+\frac{\delta}{r_\delta}\frac{1+\linfty{u}{U}}{\delta^{2\zeta}}\right) (t_0-t),
\end{equation}
where $C_2$ is a universal constant. We will now  estimate the sizes of the coefficients of the error terms on the right-hand side of the previous line. We first observe that, by the definition of $r_\delta$, there exists a constant $C_3$ that depends on $n$, $\linfty{u}{U}$ and $\linfty{Du}{U}$ so that
\begin{equation}
\label{estrdel}
\delta^\zeta<r_\delta \leq C_3\delta^{\zeta}.
\end{equation}
Therefore, 
\begin{equation}
\label{estM}
M\leq C_4 \bar{c}^{-1/\sigma}\delta^{-\frac{\zeta}{\sigma}(4n+3+2\sigma)}+M_0,
\end{equation}
where $C_4$ depends on $n$, $\lambda$, $\Lambda$, $\diam \Omega'$, $T$, $\linfty{u}{U}$, and $\linfty{Du}{\U}$.  Let us choose 
\begin{equation}
\label{choicec}
\bar{c}= \left(C_2C_4\right)^{\sigma}.
\end{equation}
We claim that, with this choice of $\bar{c}$, we have the following two bounds on the coefficients of the error terms in (\ref{touch u+ belowone}):
\begin{equation}
\label{boundxterms}
C_2 M\frac{\delta}{r_\delta}\leq (1+C_2M_0)\delta^{1/2},
\end{equation}
and,
\begin{equation}
\label{boundtterms}
\frac{\delta}{r_\delta}\frac{1+\linfty{u}{U}}{\delta^{2\zeta}} \leq \delta^{1/2}.
\end{equation}
Once these bounds are established, we may use them to bound the right-hand side of (\ref{touch u+ belowone}) from below, and then take $\bar{C}=2+C_2M_0$ to complete the proof of the lemma.

Let us first prove that (\ref{boundxterms}) holds. Since $r_\delta>\delta^{\zeta}$, we have
\[
\frac{\delta}{r_\delta}\leq \delta^{1-\zeta}.
\]
We use this, together with (\ref{estM}), to obtain,
\[
M\frac{\delta}{r_\delta}\leq C_4\bar{c}^{-1/\sigma}\delta^{1-\zeta -\frac{\zeta}{\sigma}(4n+3+2\sigma)} +M_0\delta^{1-\zeta} =C_4\bar{c}^{-1/\sigma}\delta^{1 -\frac{\zeta}{\sigma}(4n+3+3\sigma)} +M_0\delta^{1-\zeta}.
\]
Using our choice of $\zeta$ to simplify the exponent in the first term of the previous line yields,
\[
M\frac{\delta}{r_\delta}\leq C_4\bar{c}^{-1/\sigma}\delta^{1/2} +M_0\delta^{1-\zeta}.
\]
In addition, since $\zeta \leq 1/6$ (which is clear from the definition of $\zeta$), we have $1-\zeta\geq 5/6\geq 1/2$, so that $\delta^{1-\zeta}\leq \delta^{1/2}$. Therefore, we find
\[
M\frac{\delta}{r_\delta}\leq ( C_4\bar{c}^{-1/\sigma} +M_0)\delta^{1/2}.
\]
Multiplying by $C_2$ and using our choice of $\bar{c}$ yields the estimate (\ref{boundxterms}). We will now establish (\ref{boundtterms}). 
 We have $r_\delta>(1+\linfty{u}{U})\delta^{\zeta}$. Thus, we obtain
\[
\frac{\delta}{r_\delta}\frac{1+\linfty{u}{U}}{\delta^{2\zeta}}
<
\frac{\delta}{(1+\linfty{u}{U})\delta^\zeta}\frac{1+\linfty{u}{U}}{\delta^{2\zeta}}
 = \delta^{1-3\zeta}.
\]
Next, we  use  $\zeta\leq 1/6$ to bound $\delta^{1-3\zeta}$ from above by $\delta^{1/2}$ and obtain  (\ref{boundtterms}). 
The proof of the lemma is therefore complete.
\end{proof}

\begin{proof}[Proof of Lemma \ref{lem:helper3}]
We recall the representation formula for the upper monotone envelope of a function (Lemma \ref{lem:representation}):
\[
\bar{\Gamma}(w)(x_0,t_0) = \inf \{\zeta \cdot x_0 +h : \  \zeta \cdot x + h \geq w(x,t) \text{ for all } (x,t)\in Y^{-}_{2\rho}(x_0, b-r_\delta^2)\cap \{t\leq t_0\}\}.
\]
 Since $(x_0, t_0)$ is in the contact set of $u_{\delta^\zeta}^{+} -\delta^{1/4}t-c_1\delta^{\zeta} - v $ with its upper monotone envelope, the representation formula implies that there exist $\zeta\in \rr^n$ and $h\in \rr$ such that 
 \[
 \zeta \cdot x_0+h = \bar{\Gamma}(w)(x_0,t_0)=u_{\delta^\zeta}^{+}(x_0,t_0) -\delta^{1/4}t_0-c_1\delta^{\zeta} - v (x_0,t_0)
 \]
 and
\[
\zeta \cdot x +h\geq u_{\delta^\zeta}^{+}(x,t) -\delta^{1/4}t-c_1\delta^{\zeta} - v (x,t)
\]
for all $x$ and for all $t\leq t_0$ (so in particular, for all points in $Y^-_\delta(x_0,t_0)$). Rearranging the previous inequality we find,  for all $(x,t)\in Y^-_\delta(x_0,t_0)$, 
\[
v(x,t) \geq u_{\delta^\zeta}^{+}(x,t) -\delta^{1/4}t-c_1\delta^{\zeta} - \zeta \cdot x -h,
\]
with equality holding at $(x_0,t_0)$. Next, we use that $u_{\delta^\zeta}^{+}$ is touched from below at $(x_0,t_0)$ by $\phi(x,t)$ (this is assumption (\ref{touch u at x0t0phi}) of this lemma) to bound the first term on the right-hand side of the previous inequality from below. Thus we find, for  all $(x,t)\in Y^-_\delta(x_0, t_0)$, 
\begin{equation}
\label{vbelow}
v(x,t) \geq u_{\delta^{\zeta}}^+(x_0,t_0) + \phi(x,t) -\delta^{1/4}t-c_1\delta^{\zeta} - \zeta \cdot x -h,
\end{equation}
with equality at $(x_0,t_0)$. By assumption, we have $Y^-_\delta(x_0, t_0)\subset U$. Since $v$ is a $\delta$-supersolution of (\ref{eqn:mainLemma}) on $U$ and (\ref{vbelow}) holds on $Y^-_\delta(x_0, t_0)$ with equality at $(x_0,t_0)$, we find
\[
\phi_t(x_0,t_0) - \delta^{1/4} - F\left(D^2\phi(x_0,t_0)\right)\geq 0.
\]
Thus the proof of the lemma is complete.
\end{proof}

\section{Error estimate}
\label{section: comparison}
In this section we give the precise statement and the proof of our main result.
\begin{thm}
\label{main result}
Assume  (F\ref{ellipticity}), (F\ref{last prop of F}) and (U1). Assume $u\in C^{0,\eta}(\Omega\times (0,T))\cap C_x^{0,1}(\Omega \times (0,T))$ is a solution of 
\begin{equation}
\label{maineqn}
u_t-F(D^2u)=0 \text{ in } \Omega\times (0,T)
\end{equation}
and assume $\{v_{\delta}\}_{\delta>0}$ is a family of $\delta$-supersolutions  (resp. $\delta$-subsolutions) of (\ref{maineqn}) such that, for some $M<\infty$ and for all $\delta>0$, 
\begin{equation}
\label{asm:mainthm}
\etanorm{v_\delta}{\Omega\times (0,T)}\leq M
\end{equation}
and
\[
u-v_\delta\leq 0  \text{ (resp. }v_\delta-u \leq 0\text{) on }\bdry(\Omega \times (0,T)).
\]
There exist a universal constant $\bar{\delta}$, a constant $\bar{\alpha}$ that depends on $n$, $\lambda$, $\Lambda$ and $\eta$,  and a constant  $\bar{c}$ that depends on $n$, $\lambda$, $\Lambda$, $\diam\Omega$, $T$, $M$, $\etanorm{u}{\Omega \times (0,T)}$, and $\linfty{Du}{\Omega\times (0,T)}$, such that  for all $\delta\leq \bar{\delta}$,
\begin{equation}
\label{eqmainresult}
\sup_{\Omega\times (0,T)}u-v_\delta \leq  \bar{c}\delta^{\bar{\alpha}} \   \     (\text{resp. } \sup_{\Omega\times (0,T)} v_\delta-u \leq \bar{c}\delta^{\bar{\alpha}}).
\end{equation}
\end{thm}

\begin{rem}
\label{remAboutLips}
The assumption $u\in C^{0,\eta}(\Omega\times (0,T))\cap C_x^{0,1}(\Omega \times (0,T))$ is satisfied in the following situation. Assume  (F\ref{ellipticity}), (F\ref{last prop of F}), (U1), and that $\partial \Omega$ is sufficiently regular. Take $g\in C^{1,\alpha}(\bdry (\Omega\times (0,T)))$. Then, according to the interior regularity estimates of Theorem \ref{thm:Wang's c alpha thm} and the boundary estimates of \cite[Section 2]{WangII}, the solution $u$ of the boundary value problem
\begin{equation*}\left\{
\begin{array}{l l}
u_t-F(D^2 u)=0 &\quad \text{ in }\Omega\times(0,T),\\
u=g &\quad \text{ on }\bdry (\Omega\times (0,T)),
\end{array}\right.
\end{equation*}
 is indeed Lipschitz continuous in $x$ and H\"older continuous in $t$ on all of $\Omega\times (0,T)$, with 
 \[
 \etanorm{u}{\Omega\times (0,T)}, \linfty{Du}{\Omega\times(0,T)} \leq C,
 \]
 where $C$ depends on $n$, $\lambda$, $\Lambda$,  $||g||_{C^{1,\alpha}(\bdry( \Omega\times (0,T)))}$, $\diam \Omega$ and the regularity of $\partial \Omega$.
\end{rem}
\begin{proof}[Proof of Theorem \ref{main result}]
We give the proof of the bound on $\sup (u-v_\delta)$ in the case that $v_\delta$ is a $\delta$-supersolution; the proof of the other case is very similar. Let us take $\delta\leq \tilde{\delta}$, where $\tilde{\delta}$ is the universal constant given by Proposition \ref{main lemma}. To simplify notation, throughout the remainder of this proof we will use $c$ and $c_i$ with $i=0,1,2,...$, to denote positive constants that  depend only on $n$, $\lambda$, $\Lambda$, $\diam\Omega$, $T$, $M$, $\etanorm{u}{\Omega \times (0,T)}$, and $\linfty{Du}{\Omega\times(0,T)}$.

We regularize $v_\delta$ by taking inf-convolution:  let $v^-$ denote $(v_\delta)^-_{\delta^\zeta, \delta^\zeta}$ and let $U$ denote $U^{\delta^\zeta,\delta}$. For the convenience of the reader, we write down  the definition of $U$ explicitly:
\[
U= \left\{(x,t)\in \Omega\times (0,T):\  \   d_e((x,t),\bdry(\Omega \times (0,T))) \geq 2\delta^{\zeta/2}\linfty{v_\delta}{\Omega \times (0,T)}^{1/2} + \delta \right\}.
\]
We will apply part (1) of Proposition \ref{main lemma} to $v^-$ in $U$ once we verify its hypotheses. First, we see that $U$ is of the form $\Omega'\times (a,b)$, where $\Omega'$ satisfies (U1) and $(a,b)\subset (0,T)$. Next, we use items (\ref{item:inf conv is semiconcave}) and (\ref{item:lip bd on v}) of Proposition \ref{properties of inf-convolution}  with $\theta=\delta^{\zeta}$ and find that we have $D^2v^- \leq \delta^{-\zeta}I$ in  $U$ in the sense of distributions and $\tlipsemi{v^-}{U}\leq 3T\delta^{-\zeta}$. Therefore, $v^-$ satisfies assumptions (\ref{assump:d2v}) and (\ref{assump:vt}) of Proposition \ref{main lemma}. Finally,  item (\ref{inf conv is delta soln}) of Proposition \ref{properties of inf-convolution}  says that $v^-$ is a $\delta$-supersolution of  (\ref{eqn:mainLemma}) in $U$. 

Now let $r_\delta$ and $\tilde{U}$ be as in the statement of Proposition \ref{main lemma}:
\begin{align*}
r_\delta &= 9\sqrt{n}(1+\linfty{u}{U}+2\linfty{Du}{U})\delta^{\zeta}, \text{ and }\\
\tilde{U} &= \left\{(x,t)\in U:\  \   d((x,t),\bdry  U) \geq  2r_\delta, \text{ and }t\leq T-r_\delta^2\right\}.
\end{align*}
Let us  apply Proposition \ref{main lemma} with $u-\left(\sup_{\bdry \tilde{U}} (u-v^-)\right)$ instead of just $u$ itself (this ensures that the  hypothesis (\ref{main prop wlog}) is satisfied.) We conclude
\begin{equation}
\label{consequence of Propo}
\sup_{\tilde{U}}(u-v^-) \leq \tilde{c}\delta^{\tilde{\alpha}}+ \sup_{\bdry \tilde{U}}(u-v^-).
\end{equation}
We will now use this  estimate, the regularity properties of $u$ and $v_\delta$, and the fact that the distance between $\bdry \tilde{U}$ and $\bdry (\Omega\times (0,T))$ is small to deduce the bound (\ref{eqmainresult}).

First, we use the definitions of $U$, $\tilde{U}$ and $r_\delta$ to estimate  the distance between $\bdry \tilde{U}$ and $\bdry (\Omega\times (0,T))$: 
\begin{equation}
\label{distances}
d( \tilde{U}, \bdry(\Omega\times(0,T)))=\inf_{(x,t)\in \tilde{U}}\{d((x,t),\bdry (\Omega\times(0,T))\}\leq 
2r_\delta+2\delta^\frac{\zeta}{2}\linfty{v_\delta}{\Omega\times(0,T)}^{1/2}+\delta\leq c_0 \delta^{\zeta/2},
\end{equation}
where the last inequality follows from the definition of $r_\delta$ and since $\frac{\zeta}{2}<\zeta<1$. 

Next, since $v^-\leq v_\delta$ holds on $\Omega\times (0,T)$, and we have $u\in C^{0,\eta}(\Omega\times (0,T))$, $v_\delta\in  C^{0,\eta}(\Omega \times (0,T))$, and $u\leq v_\delta$ on $\bdry (\Omega\times (0,T))$, the estimate (\ref{distances}) implies that there exists a  constant $c_1$ such that
\begin{equation}
\label{bound on bdry U}
\sup_{\bdry \tilde{U}} (u-v_\delta)\leq c_1 \delta^{\frac{\zeta\eta}{2}}
\end{equation}
and 
\begin{equation}
\label{bound on whole set}
 \sup_{\Omega \times (0,T)} (u- v_\delta) \leq \sup_{\tilde{U}} (u-v^-) +c_1 \delta^{\frac{\zeta\eta}{2}}.
\end{equation}

We use (\ref{consequence of Propo}) to bound the first term on the right-hand side of (\ref{bound on whole set}) from above and obtain,
\begin{equation}
\label{bdwholesetalmost}
\sup_{\Omega\times (0,T)} (u-v_\delta)\leq \tilde{c}\delta^{\tilde{\alpha}}+ \sup_{\bdry \tilde{U}}(u-v^-)+c_1 \delta^{\frac{\zeta\eta}{2}}.
\end{equation}
Thus, to establish (\ref{eqmainresult}) it is left to bound the right-hand side of (\ref{bdwholesetalmost}) from above. To this end, item (\ref{item:bound from below on inf conv}) of Proposition \ref{properties of inf-convolution} applied with $\theta=\delta^{\zeta}$, and the assumption (\ref{asm:mainthm}) of this theorem, imply, for all $(x,t)\in U$,
\[
v^-(x,t) \geq v_\delta(x,t)- \etasemi{v_\delta}{\Omega \times (0,T)}^{\frac{2}{2-\eta}} \delta^{\frac{\zeta\eta}{2-\eta}} \geq v_\delta(x,t)- M^{\frac{2}{2-\eta}} \delta^{\frac{\zeta\eta}{2-\eta}}.
\]
We apply this estimate with  $(x,t)\in \bdry\tilde{U}\subset U$ to obtain a bound from above for the second term on the right-hand side of (\ref{bdwholesetalmost}): 
\[
\sup_{\bdry \tilde{U}}(u-v^-)\leq \sup_{\bdry \tilde{U}}(u - v_\delta) +M^{\frac{2}{2-\eta}} \delta^{\frac{\zeta\eta}{2-\eta}}.
\]
We now use (\ref{bound on bdry U}) to bound the first term on the right-hand side of the previous line and obtain
\[
\sup_{\bdry \tilde{U}}(u-v^-) \leq c_1 \delta^{\frac{\zeta\eta}{2}}+ M^{\frac{2}{2-\eta}} \delta^{\frac{\zeta\eta}{2-\eta}}\leq c_2\delta^{\frac{\zeta\eta}{2}} .
\]
Next we use the previous line to bound the second term on the right-hand side of (\ref{bdwholesetalmost}) from above, and obtain
\[
 \sup_{\Omega \times (0,T)} (u- v_\delta) \leq\tilde{c}\delta^{\tilde{\alpha}} +c_2\delta^{\frac{\zeta\eta}{2}} +c_1\delta^{\frac{\zeta\eta}{2}}.
\]
We thus take $\bar{c}=\tilde{c}+c_2+c_1$ and $\bar{\alpha}=\min\{\tilde{\alpha}, \frac{\zeta\eta}{2}\}$ to complete the proof of the theorem.
\end{proof}

\section{Discrete approximation schemes}
\label{section: approx}
We now present the error estimate for finite difference approximation schemes. First, we introduce the necessary notation and assumptions, closely following \cite{Approx schemes} and \cite{Kuo Trudinger parabolic diff ops}. In the next section we give the full statement of the error estimate and its proof. The space-time mesh is denoted by $E$:
\[
E=h \zz^n \times h^2 \zz=\{(x,t):\  \   \text{$x=(m_1,..,m_n)h$, $t=mh^2$,  where $ m,m_1,...,m_n\in \zz$}\}.
\]
We fix some $N>1$ and define the subset $Y$ of $E$ by 
\[
Y= \{ y\in h\zz^n:\  \    0<|y|<hN\}.
\]
Next we introduce finite difference operators for a function $u$:
\begin{align*}
\delta^-_{\tau} u (x,t) &= \frac{1}{h}(u(x,t) - u(x,t-h^2)),\\
\delta_y^2 u (x,t)&=\frac{1}{|y|^2}(u(x+y,t)+u(x-y,t)-2u(x,t)), \text{ and}\\
\delta^2u(x,t) &=\{\delta^2_y u (x,t):\  \   \   y\in Y\}.
\end{align*}
An \emph{implicit finite difference operator}  is an operator of the form
\[
\Ph [u] (x,t) = \delta^-_{\tau} u (x,t) - \mathcal{F}_h(\delta^2 u(x,t)),
\]
where $\mathcal{F}_h: \rr^Y \rightarrow \rr$ is locally Lipschitz. We denote points in $\rr^Y$ by $r=(r_1,...,r_{|Y|})$. We say an operator  is \emph{monotone} if it satisfies:
\begin{enumerate}[({S}1)] 
\item there exists a constant $\lambda_0$ and $\Lambda_0$ such that for all $i=1,...,|Y|,$
\[
\lambda_0\leq \frac{\partial \mathcal{F}_h}{\partial r_i}\leq \Lambda_0.
\]
\end{enumerate}
This is equivalent to the definition given in the introduction. 

A scheme $\Ph$ is said to be \emph{consistent with $F$} if for all $\phi \in C^3(\Omega\times(0,T))$,
\[
\sup|\phi_t-F(D^2\phi)-\Ph[\phi]| \rightarrow 0 \text{ as } h\rightarrow 0.
\]
In  \cite[Section 4]{Kuo Trudinger parabolic diff ops}, it was shown that if a nonlinearity $F$ satisfies (F\ref{ellipticity}), then there  exists a monotone implicit scheme  $\Ph$ that is consistent with $u_t-F(D^2u)=0$, and the constants $\lambda_0$ and $\Lambda_0$ depend only on $n$, $\lambda$ and $\Lambda$. In \cite{Kuo Trudiner elliptic ops}, a monotone and consistent approximation scheme for elliptic equations is explicitly constructed, and the construction in the parabolic case is analogous. 

In order to obtain an error estimate, we need to make an assumption that quantifies the above rate of convergence. As in \cite{Approx schemes}, we assume:
\begin{enumerate}[({S}2)] 
\item
there exists a positive constant $K$ such that for all $\phi\in C^3(\Omega\times(0,T))$,
\[
\sup|\phi_t-F(D^2\phi)-\Ph[\phi]|\leq K(h+ h\linfty{D_x^3\phi}{\Omega\times (0,T)} + h^2\linfty{\phi_{tt}}{\Omega\times (0,T)}).
\]
\end{enumerate} 
Schemes that satisfy (S\ref{consistent}) are said to be \emph{consistent with an error estimate for $F$ in $\Omega\times (0,T)$ with constant $K$}. 

Given $\Omega\times (0,T)\subset \rr^{n+1}$, we denote the mesh points $ (\Omega\times (0,T)) \cap E$ by $\Umesh$.  We divide $\Umesh$ into  interior and boundary points relative to the operator $\Ph$. We define
\[
\Umesh^i=\left\{p\in \Umesh :\  \   \    d(p, \bdry\Omega \times (0,T))\geq Nh\right\}
\]
to be the interior points and
\[
\Umesh^b = \Umesh \setminus \Umesh^i
\]
to be the boundary points. Notice that $\Ph[u](x,t)$ depends only  on $u(x+y,s)$ for $0\leq|y|<hN$ and for $t-h^2\leq s\leq t$.

The discrete H\"older seminorm of a  function  $u$ on $\Umesh$  is defined to be
\[
\etasemi{u}{\Umesh} = \sup_{p,q \in \Umesh}\frac{|u(p)-u(q)|}{d(p,q)^{\eta}}
\]
and the discrete  H\"older norm is
\[
\etanorm{u}{\Umesh} = \linfty{u}{\Umesh}+\etasemi{u}{\Umesh} .
\]

In \cite[Section 4]{Kuo Trudinger parabolic diff ops} it is shown that solutions of the discrete equation $\Ph[v_h] = 0$ are uniformly equicontinuous. We summarize this result:
\begin{thm}
\label{KTthm}
Assume that $\Ph$ is an implicit monotone finite difference scheme and that $v_h$ is a solution of $\Ph[v_h] = 0$. There exists a constant $C$ that depends only on $n$, $\lambda_0$ and $\Lambda_0$ such that for all $h\in (0,1)$,
\[
\etanorm{v_h}{\Umesh^i} \leq C.
\]
\end{thm}

\subsection{Inf and sup convolutions for approximation schemes}
\begin{defn} 
Given $\theta>0$ and a mesh function $v: \Umesh\rightarrow \rr$, we define the \emph{inf- convolution} $\convmesh$ and the \emph{sup- convolution} $v^+_{\theta,\theta}$ of $v$ at $(x,t)\in \Omega\times (0,T)$ by
\[
\convmesh(x,t) = \inf_{(y,s)\in \Umesh} \left\{ v(y,s) + \frac{|x-y|^2}{2\theta}+\frac{|t-s|^2}{2\theta}\right\}
\]
and
\[
v^{+}_{\theta, \theta}(x,t) = \sup_{(y,s)\in \Umesh} \left\{ v(y,s) - \frac{|x-y|^2}{2\theta}-\frac{|t-s|^2}{2\theta}\right\}.
\]
\end{defn}
\begin{defn}
Given $\theta>0$ and $v\in C^{0,\eta}(\Umesh)$, we introduce  the quantity
\[
\omega(h, \theta)=nh+2\theta^{1/2}\linfty{v}{\Umesh}
\] 
and the set
\[
U^h_{\theta} = \{ p\in \Omega \times (0,T) :\  \    d_e(p,\bdry{(\Omega \times (0,T))}) \geq \omega(h, \theta) +Nh\}.
\]
\end{defn}
In the appendix we summarize the basic properties of inf- and sup- convolutions of mesh functions  (see Proposition \ref{prop: inf-conv for mesh}).

It is a classical fact of the theory of viscosity solutions that if $u$ is the viscosity solution of $u_t-F(D^2u)=0$ in $\Omega\times (0,T)$, then the sup-convolution of $u$ is a subsolution of the same equation. In the following proposition, we establish a similar relationship between solutions $v_h$ of discrete equations and $\delta$-solutions of  $u_t-F(D^2u)=0$. This is a key step in establishing an error estimate for discrete approximation schemes.
\begin{prop}
\label{item: mesh solves equation} 
Assume that $\Ph$ is a monotone and implicit scheme consistent with an error estimate for $F$ with constant $K$. Assume $v\in C^{0,\eta}(\Umesh)$.
\begin{enumerate}
\item If $v$ satisfies $\Ph[v]  \geq 0$ in $ \Umesh^i$ then  $\convmesh$ is a $\delta$-supersolution of 
\[
u_t - F(D^2u) = - Kh
\]
 in $U^h_{\theta}$, with $\delta= Nh$.
 \item If  $v$ satisfies $\Ph[v]  \leq 0$ in $ \Umesh^i$, then $v^{+}_{\theta, \theta}$ is $\delta$-subsolution of 
 \[
 u_t - F(D^2u) =  Kh
 \]
  in $U^h_{\theta}$, with  $\delta= Nh$.
\end{enumerate}

\end{prop}

\begin{proof}
We give the proof of item (1); the proof of item (2) is analogous.

Let us suppose $(x,t)\in U^h_\theta$ is such that $Y^-_{Nh}(x,t)\subset U^h_\theta$ and  that  $P\in \PPi$ is such that 
\begin{equation}
\label{Ptouchvh}
P\leq \convmesh \text{ on }Y^-_{Nh}(x,t), \text{ and } P(x,t)=\convmesh(x,t).
\end{equation}
We use (\ref{Ptouchvh}) and the definition of the operators $\delta^2_y$ and $\delta^-_\tau$ to obtain, for each $y$ with $0<|y|\leq hN$,
\begin{equation}
\label{eq:implicita}
\delta^2_yP(x,t) \leq \delta^2_y\convmesh(x,t), \text{ and } \delta^-_\tau P(x,t) \geq \delta^-_\tau \convmesh(x,t).
\end{equation}
(We remark that for this to hold, it is essential that $P\leq v$ on \emph{all} of $Y^-_{Nh}(x,t)$.) Now let $(x^*, t^*)$ be a point at which the infimum is achieved in the definition of $v^-_{\theta,\theta}(x,t)$. According to item (\ref{opsandconv}) of Proposition \ref{prop: inf-conv for mesh}, we have 
\begin{equation}
\label{deltash}
\delta^2_y\convmesh(x,t)\leq \delta^2_y v(x^*,t^*) \text{ and } \delta^-_\tau \convmesh(x,t) \geq \delta^-_\tau v(x^*, t^*).
\end{equation}
We now use the first inequality in (\ref{deltash}) to estimate the right-hand side of the first inequality  in (\ref{eq:implicita}) from above, and we use the second inequality in (\ref{deltash}) to estimate the right-hand side of the second inequality  in (\ref{eq:implicita}) from below. We find,
\begin{equation}
\label{eq:implicit}
\delta^2_yP(x,t) \leq \delta^2_y v(x^*,t^*), \text{ and } \delta^-_\tau P(x,t) \geq  \delta^-_\tau v(x^*, t^*).
\end{equation}
Item (\ref{eq: d(q,p^*)}) of Proposition \ref{prop: inf-conv for mesh} yields $(x^*,t^*)\in \Umesh^i$. Since $v$ satisfies $\Ph[v]\geq 0$ in $\Umesh^i$, we obtain, 
\[
0\leq \Ph[v](x^*, t^*) = \delta^-_\tau v(x^*, t^*)  - \mathcal{F}_h(\delta^2 v(x^*, t^*))  .
\]
Using (\ref{eq:implicit}) and the monotonicity of $\mathcal{F}_h$ to bound the right-hand side of the previous line yields
\begin{equation}
\label{eq:scheme at x*}
0\leq \delta^-_\tau P(x,t) - \mathcal{F}_h(\delta^2P(x,t)).
\end{equation}
We use that $\Ph$ is consistent with an error estimate and find,
\[
|P_t(x,t) -F(D^2P(x,t))- (\delta^-_\tau P(x,t)   - \mathcal{F}_h(\delta^2 P(x,t) ))|
\leq Kh.
\]
Together with (\ref{eq:scheme at x*}), this implies
\[
P_t -F(D^2P) \geq -Kh,
\]
as desired.
\end{proof}

\section{Error estimate for approximation schemes}
\label{section:error est approx}
In this section we give the precise statement and proof of Theorem \ref{thm:mesh comparison vague}.
\begin{thm}
\label{prop: convergence for approx schemes}
Assume (F\ref{ellipticity}), (F\ref{last prop of F}) and (U1). 
Assume $u\in C^{0,\eta}(\Omega\times (0,T))\cap C_x^{0,1}(\Omega \times (0,T))$ is a solution of
\begin{equation}
\label{h thm main eqn}
u_t-F(D^2u)=0
\end{equation}
in $\Omega \times (0,T)$.  Assume that $\Ph$ is an implicit monotone  scheme consistent with an error estimate for (\ref{h thm main eqn}) and that $v_h$ satisfies $\Ph[v_h] \geq 0$ (resp. $\Ph[v_h] \leq 0$) in $\Umesh$ for all $h>0$. Assume that for some constant $M<\infty$ and for all $h>0$,
\[
\etanorm{v_h}{\Umesh^i}\leq M
\]
and
\begin{equation}
\label{hassum}
u-v_h\leq 0 \text{ (resp. } u-v_h\geq 0\text{) on }\Umesh^b.
\end{equation}
There exists a constant $\bar{h}$ that depends only on $\lambda$, $\Lambda$, $n$ and $N$, a constant $\bar{\alpha}$ that depends on $n$, $\lambda$, $\Lambda$ and $\eta$,  and a constant  $\bar{c}$ that depends on $n$, $N$, $\lambda$, $\Lambda$, $\diam\Omega$, $T$, $K$, $M$, $\etanorm{u}{\Omega\times (0,T)}$, and $\linfty{Du}{\Omega\times(0,T)}$, such that for all $h\leq \bar{h}$,
\begin{equation*}
\sup_{\Umesh}u-v_h \leq  \bar{c}h^{\bar{\alpha}} \   \     (\text{resp. } \sup_{\Umesh}v_h-u \leq \bar{c}h^{\bar{\alpha}}).
\end{equation*}
\end{thm}

\begin{rem}
 Suppose that $u$ is the solution of the boundary value problem
\begin{equation*}\left\{
\begin{array}{l l}
u_t-F(D^2 u)=0 &\quad \text{ in }\Omega\times(0,T),\\
u=g &\quad \text{ on }\bdry (\Omega \times (0,T)),
\end{array}\right.
\end{equation*}
and $v_h$ satisfies
\begin{equation*}\left\{
\begin{array}{l l}
\Ph[v_h]=0 &\quad \text{ in }\Umesh^i,\\
v_h=g &\quad \text{ on }\Umesh^b,
\end{array}\right.
\end{equation*}
where $\partial \Omega$ is sufficiently regular and  $g\in C^{1,\alpha}( \Omega\times (0,T))$. Then, as explained in Remark \ref{remAboutLips}, $u$ satisfies the hypotheses of Theorem \ref{prop: convergence for approx schemes}, and $v_h$ satisfies the hypotheses of Theorem \ref{prop: convergence for approx schemes} because of  Theorem \ref{KTthm}.
\end{rem}

\begin{proof}[Proof of Theorem \ref{prop: convergence for approx schemes}]
We give the proof of the bound on $\sup (u-v_h)$; the proof of the other case is very similar. We will apply Proposition \ref{main lemma} with $\delta=Nh$. We define $\bar{h}$ to be $\bar{h}=\tilde{\delta}N^{-1}$, where $\tilde{\delta}$ is the universal constant given by Proposition \ref{main lemma}. To simplify notation, throughout the remainder of this proof we will use $c$ and $c_i$ with $i=0,1,2,...,$ to denote positive constants that depend only on $n$, $N$, $\lambda$, $\Lambda$, $\diam\Omega$, $T$, $K$, $M$, $\etanorm{u}{\Omega\times (0,T)}$, and $\linfty{Du}{\Omega\times(0,T)}$. In addition, $c$ may change from line to line. 

First we regularize $v_h$ by taking inf-convolution: we denote $v^-= (v_h)^-_{h^{\zeta}, h^\zeta}$ and $U^h_{h^{\zeta}}$ by $U$. For the convenience of the reader, we write down the definition of $U$ explicitly:
\[
U= \left\{(x,t)\in \Omega\times (0,T) :\  \    d((x,t),\bdry(\Omega \times (0,T))) \geq \omega(h,h^\zeta)+Nh\right\},
\]
and
\[
\omega(h, h^\zeta)=nh+2h^{\zeta/2}\linfty{v_h}{\Umesh}.
\] 
Then, according to item (\ref{inf conv is delta soln}) of Proposition (\ref{item: mesh solves equation}),  $v^-$ is a $\delta$-supersolution of 
\[
v^-_t-F(D^2v^-)\geq -Kh
\]
in $U$. We point out that $U$ is of the form $\Omega'\times (a,b)$, where $\Omega'$ satisfies (U1). Therefore, Proposition \ref{main lemma} holds in $U$. Since  Proposition \ref{main lemma} applies to $\delta$-solutions with right-hand side $0$, we have to perturb $v^-$. We introduce 
\[
v(x,t)=v^-(x,t)+Kht,
\]
so that $v$ is a $\delta$-supersolution of (\ref{h thm main eqn}) in $U$. Moreover, by items  (\ref{item: mesh inf conv is semiconcave}) and (\ref{item: mesh lip bd on v}) of Proposition \ref{prop: inf-conv for mesh}, $v$ satisfies the regularity assumptions  (\ref{assump:d2v}) and (\ref{assump:vt}) of Proposition \ref{main lemma}. Let $r_\delta$ and $\tilde{U}$ be as in Proposition \ref{main lemma} (recall we are taking $\delta=Nh$):
\begin{align*}
r_\delta &= 9\sqrt{n}(1+\linfty{u}{U}+2\linfty{Du}{U})(Nh)^{\zeta}, \text{ and }\\
\tilde{U} &= \left\{(x,t)\in U :\  \    d((x,t),\bdry  U) \geq  2r_\delta, \text{ and }t\leq T-r_\delta^2\right\}.
\end{align*}
We apply Proposition \ref{main lemma} to with $u-\left(\sup_{\bdry \tilde{U}}(u-v)\right)$ instead of $u$, so that  the hypothesis  (\ref{main prop wlog}) of Proposition (\ref{main lemma}) is satisfied. We obtain the bound
\begin{equation*}
\sup_{\tilde{U}}(u-v) \leq c(Nh)^{\tilde{\alpha}}+ \sup_{\bdry \tilde{U}}(u-v).
\end{equation*}
By the definition of $v$, we have $v^-\leq v\leq v^-+KhT$ on $\tilde{U}$. We use this to bound the left-hand side of the previous line from below and the right-hand side of the  previous line from above and find,
\begin{equation}
\label{h consequence of Propo}
\sup_{\tilde{U}}(u-v^-) \leq KhT+ c(Nh)^{\tilde{\alpha}}+ \sup_{\bdry \tilde{U}}(u-v^-)\leq c_1h^{\tilde{\alpha}}+ \sup_{\bdry \tilde{U}}(u-v^-),
\end{equation}
where the second inequality follows since $\tilde{\alpha}\leq 1$.
We will now deduce the conclusion of the theorem from (\ref{h consequence of Propo}). This step is a bit more complicated than the corresponding step in the proof of Theorem \ref{main result} because $v_h$ is a function on  the mesh $\Umesh$ and not on all of $\Omega\times (0,T)$. By the definitions of $U$ and $\tilde{U}$, we have the following estimate, which we'll be applying frequently in the remainder of the argument: 
\[
d(\tilde{U}, \Umesh^b)= \inf_{(x,t)\in \tilde{U}}\{d((x,t),\Umesh^b)\}\leq c(h+h^{\zeta/2}+h^\zeta)\leq c_0 h^{\zeta/2}.
\]
We first bound $\sup_{\Umesh} (u-v_h)$ from above. Since $u$ and $v_h$ are H\"older-continuous, and since  $v^-\leq v^h$ holds on $\Umesh$, we have:
\[
\sup_{\Umesh} (u-v_h)\leq \sup_{\tilde{U}\cap E} (u-v_h) + c (d(\tilde{U}, \Umesh^b))^{\eta} \leq \sup_{\tilde{U}\cap E} (u-v_h) +ch^{\frac{\zeta\eta}{2}}\leq \sup_{\tilde{U}\cap E} (u-v^-) +c_2 h^{\frac{\zeta\eta}{2}}.
\]
We next use   (\ref{h consequence of Propo}) to estimate the first term on right-hand side of the previous line from above. We find,
\begin{equation}
\label{hbdUmesh}
\sup_{\Umesh} (u-v_h)\leq  c_1 h^{\tilde{\alpha}}+ \sup_{\bdry \tilde{U}}(u-v^-) +c_2h^{\frac{\zeta\eta}{2}}.
\end{equation}
We will now establish the following upper bound for the middle term of right-hand side of (\ref{hbdUmesh}):
\begin{equation}
\label{hbdonbdry}
\sup_{\bdry \tilde{U}}(u-v^-) \leq \sup_{\Umesh^b}(u-v_h)+c_3 h^{\frac{\zeta\eta}{2}}.
\end{equation}
To this end, let us fix some $(x,t)\in \bdry \tilde{U}$. Let $(y,s)$ be a nearest mesh point to $(x,t)$ (recall that $v_h$ is only defined on the mesh $E$, and $(x,t)$ need not be a mesh point).  By item (\ref{item: mesh: bound from below on inf conv}) of Proposition \ref{prop: inf-conv for mesh}, we have a bound from below on the difference between $v^-(x,t)$ and $v_h(y,s)$:
\begin{equation}
\label{hvv-}
v^-(x,t)\geq v_h(y,s)-M\omega(h, h^\zeta)^\eta\geq v_h(y,s) - ch^{\frac{\zeta\eta}{2}} ,
\end{equation}
where the second inequality follows from the definition of $\omega$. In addition, since $(x,t)\in \bdry\tilde{U}$ is ``near" $\Umesh^b$ and $(y,s)$ is a nearest mesh point to $h$, we have that $(y,s)$ is ``near" $\Umesh^b$ too. Precisely, there exists a point $(z,r)\in \Umesh^b$ with 
\[
d((y,s), (z,r))\leq nh + d(\tilde{U}, \Umesh^b) \leq ch^{\zeta/2}.
\]
We use this estimate and the fact that $v_h$ is  H\"older-continuous to bound the difference between $v_h(y,s)$ and $v_h(z,t)$ and obtain,
\[
v_h(y,s)\geq v_h(z,t)-ch^{\frac{\zeta\eta}{2}}.
\]
Using this to bound the right-hand side of (\ref{hvv-}), we find,
\begin{equation*}
v^-(x,t)\geq v_h(z,r)-ch^{\frac{\zeta\eta}{2}} .
\end{equation*}
Using the previous estimate and the fact that $u$ is also H\"older continuous and $d(\tilde{U}, \Umesh^b)\leq  c_0 h^{\zeta/2}$, we obtain
\[
u(x,t) -v(x,t) \leq u(z,r)-v_h(z,r) +ch^{\frac{\zeta\eta}{2}} \leq \sup_{\Umesh^b}(u-v_h)+c_3h^{\frac{\zeta\eta}{2}},
\]
where the second inequality follows since $(z,r)\in \Umesh^b$.
Since this holds for any $(x,t)\in \bdry \tilde{U}$, we find that (\ref{hbdonbdry}) holds as well. 
The assumption (\ref{hassum}) says that the first term on the right-hand side of (\ref{hbdonbdry}) is non-positive.  The estimate (\ref{hbdonbdry}) therefore becomes,
\begin{equation}
\label{uv-almost}
\sup_{\bdry \tilde{U}} (u-v^-)\leq c_3 h^{\frac{\zeta\eta}{2}}.
\end{equation}
Using (\ref{uv-almost}) to bound the right-hand side of (\ref{hbdUmesh}) and taking $\bar{c}=c_1+c_2+c_3$ and $\bar{\alpha}=\min\{\tilde{\alpha},\frac{\zeta\eta}{2}\}$ completes the proof the theorem.
\end{proof}

\appendix
\section{}
\label{Appendix}
\subsection{Proof of Proposition \ref{pre ABP}}
\label{subsec:proof of pre ABP}
We outline the proof Proposition \ref{pre ABP},  which is a slight modification of part of  the proof of the parabolic version of the ABP estimate in \cite[Section 4.1.2]{ImbertSilvestre}. We assume $\rho=1$; the general case follows by  rescaling. 

As in \cite[Section 4.1.2]{ImbertSilvestre}, we define the function $G\Gamma: Y^-_2\rightarrow \rr^{n+1} $ by
\[
G\Gamma(x,t)= ( \Gamma(x,t)-x\cdot D\Gamma(x,t), D\Gamma(x,t) ).
\]
Proposition \ref{pre ABP} follows from the following three lemmas. 
\begin{lem}
\label{lemPreABPsup}
Suppose $u$ is as in Proposition \ref{pre ABP}. Let $M$ denote $\sup_{Y_1^-}u^-$. Then, 
\[
\{(h,\xi)\in \rr^{n+1} : |\xi|\leq \frac{M}{2}\leq -h\leq M\} \subset G\Gamma(Y_1^-\cap \{u=\Gamma\}).
\]
\end{lem}

\begin{lem}
\label{lemPreABPDG}
If $\Gamma$ is $C^{1,1}$ with respect to $x$ and Lipschitz with respect to $t$, then $G\Gamma$ is Lipschitz in $(x,t)$ and, for almost every $(x,t)\in Y^-_1$,
\[
\det D_{x,t}G\Gamma = \partial_t\Gamma \det D^2\Gamma.
\]
\end{lem}

\begin{lem}
\label{lemPreABPGamma}
Suppose $u$ is as in  Proposition \ref{pre ABP}. Then $\Gamma$ is $C^{1,1}$ with respect to $x$ and Lipschitz in $t$.
\end{lem}

Lemma \ref{lemPreABPsup} is very similar to  \cite[Lemma 4.13]{ImbertSilvestre} -- the latter is stated for $u$ a supersolution of $u_t-\mathcal{M}^-(D^2u)\geq f(x)$ in $Y_1^-$. However, it is easy to verify that the proof of  \cite[Lemma 4.13]{ImbertSilvestre} holds for a  general continuous  function $u$, not only for supersolutions. Lemma \ref{lemPreABPDG} is exactly \cite[Lemma 4.4]{ImbertSilvestre}. 
The proof of Lemma \ref{lemPreABPGamma} is very similar to that of Lemma \cite[Lemma 4.11]{ImbertSilvestre}, which states that if $u$ is a supersolution of $u_t-\mathcal{M}^-(D^2u)\geq f(x)$, then $\Gamma$ is $C^{1,1}$ in $x$ and Lipschitz in $t$.  We  explain the modifications needed to obtain Lemma \ref{lemPreABPGamma}:
\begin{proof}[Proof of Lemma \ref{lemPreABPGamma}] 
As in the  proof of Lemma \cite[Lemma 4.11]{ImbertSilvestre}, it will suffice to show that there exists a constant $K$ such that for all $(x,t)\in Y^-_1$, the following holds:
\begin{equation}
\label{jet condition}
\left\{
\begin{split}
&\text{if $P(y,s)=c+\alpha t +p\cdot x+ x\cdot Xx^T$ satisfies $P(x,t)=\Gamma(x,t)$ and }\\ &\text{$P(y,s)\leq \Gamma(y,s)$ for all $(y,s)$ in a neighborhood of $(x,t)$,}\\&\text{then $\alpha\geq -K$ and $X\leq KI$.}
\end{split}
\right.
\end{equation}
In the proof of \cite[Lemma 4.11]{ImbertSilvestre}, the assumption that $u$ is a supersolution of $u_t-\mathcal{M}^-(D^2u)\geq f(x)$ is used only to show that (\ref{jet condition}) holds for $(x,t)\in \{u=\Gamma\}$. Thus, we need to verify  (\ref{jet condition}) for such points. To this end, assume $(x,t)\in \{u=\Gamma\}$ and $P(y,s)=c+\alpha t +p\cdot x+ x\cdot X x^T$ satisfies $P(x,t)=\Gamma(x,t)$ and $P(y,s)\leq \Gamma(y,s)$ for all $(y,s)$ in a neighborhood of $(x,t)$. Since $\Gamma$ is the lower monotone envelope of $\min(u,0)$ and  $(x,t)\in \{u=\Gamma\}$, we have 
\begin{equation}
\label{Ptouchu}
P(x,t)=u(x,t) \text{ and }P(y,s)\leq u(y,s) \text{ for all $(y,s)$ in a neighborhood of $(x,t)$}.
\end{equation}
Since we have $D^2u(x,t)\leq KI$ in the sense of distributions,  there exists a paraboloid $Q(x)$ of opening $K$ that touches $u(\cdot,t)$ from above at $x$.  Together with (\ref{Ptouchu}), this implies that $Q$ touches $P(\cdot, t)$ from above at $x$. Therefore, $X\leq KI$. 
In addition, (\ref{Ptouchu}) implies 
\[
|P_t(x,t)|\leq \tlipsemi{u}{Y_1^-}.
\]
Since  $\tlipsemi{u}{Y_1^-}\leq K$, we obtain $\alpha = P_t(x,t) \geq -K$. Thus we've verified (\ref{jet condition}) and the proof of the lemma is complete.
\end{proof}

\begin{proof}[Proof of Proposition \ref{pre ABP}]
Lemma \ref{lemPreABPsup},  Lemma \ref{lemPreABPDG} and the Area Formula \cite[Chapter 3]{Evans Gariepy}, which applies to $G\Gamma$ because of Lemma \ref{lemPreABPGamma}, imply
\begin{equation}
\label{ABP-like bd}
\begin{split}
CM^{n+1} &= |\{(h,\xi)\in \rr^{n+1} : |\xi|\leq \frac{M}{2}\leq -h\leq M\} |\\&\leq |G\Gamma(Y_1^-\cap \{u=\Gamma\})|\\
&\leq \int_{Y_1^-\cap \{u=\Gamma\}} |\det D_{x,t} G\Gamma |\\&\leq  \int_{Y_1^-\cap \{u=\Gamma\}} -\partial_t\Gamma \det D^2\Gamma .
\end{split}
\end{equation}
Because $\Gamma$ is non-increasing in $t$ and convex in $x$, the Alexandroff theorem for functions depending on $x$ and $t$ (see Krylov \cite[Appendix 2]{Krylov book}) implies that there exists $A\subset Y_1^-$ with $|Y_1^-\setminus A|=0$ such that for every $(x,t)\in A$, 
\[
|\Gamma(y,s)-\Gamma(x,t)-D\Gamma(x,t)\cdot (y-x) - (y-x) \cdot D^2\Gamma(x,t)(y-x)^T-\partial_t\Gamma(x,t)(s-t)| \leq o(|x-y|^2+|t-s|)
\]
(in other words, $\Gamma$ is twice differentiable almost everywhere on $Y_1^-$). 
We claim
\begin{equation}
\label{claim for 2nd pt}
\text{ if $(x,t)\in A\cap \{u=\Gamma\}$, then }
-\partial_t\Gamma(x,t) \det D^2\Gamma(x,t) \leq K^{n+1}.
\end{equation}
Let us fix $(x,t)\in A\cap \{u=\Gamma\}$. We have 
\begin{equation}
\label{Gamtouchu}
u(x,t)=\Gamma(x,t)\text{ and }u(y,s)\geq \Gamma(y,s)\text{ for all $(y,s)$ in a neighborhood of $(x,t)$}.
\end{equation}
Since we have $D^2u(x,t)\leq KI$ in the sense of distributions, we find that  there exists a paraboloid $Q(x)$ of opening $K$ that touches $u(\cdot,t)$ from above at $x$.  Therefore, $Q$ touches $\Gamma(\cdot, t)$ from above at $x$. Thus, $D^2\Gamma(x,t)\leq KI$. Since $\Gamma$ is convex in $x$, we obtain $D^2\Gamma(x,t)\geq 0$. Therefore,
\[
0\leq\det D^2\Gamma \leq K^n.
\] 
In addition, (\ref{Gamtouchu}) implies 
\[
|\partial_t\Gamma(x,t) |\leq \tlipsemi{u}{Y_1^-}.
\]
Since $\tlipsemi{u}{Y_1^-}\leq K$, we obtain $\partial_t\Gamma(x,t) \geq -K$. Since $\Gamma$ is non-increasing in $t$, we have $\partial_t\Gamma(x,t)\leq 0$. Together with the estimate on $\det D^2\Gamma$, this implies that (\ref{claim for 2nd pt}) holds. The bound (\ref{cons of ABPbd}) follows by using (\ref{claim for 2nd pt}) to bound the right-hand side of (\ref{ABP-like bd}), recalling the definition of $M=\sup_{Y_1^-} u^-$, and rearranging. 

\end{proof}

\subsection{Inf and sup convolutions}
In the following proposition we state the facts about inf- and sup- convolutions that are used in this paper. Their proofs are very similar to those in the elliptic case (see \cite[Propositions 5.3 and 5.5]{Rates Homogen} and \cite[Lemma 5.2]{Cabre Caffarelli book}) and we omit them.  
\begin{prop}
\label{properties of inf-convolution}
Assume $v\in C(\Omega\times (0,T))$.  Then:
\begin{enumerate}
\item We have $v^-_{\theta, \theta}(x,t)\leq v(x,t)\leq v^+_{\theta, \theta}(x,t)$ for all $(x,t)\in \Omega\times (0,T)$.
\item 
\label{item:bound from below on inf conv}
If $v\in C^{0,\eta}(\Omega \times (0,T))$, then for any $(x,t)\in U^{\theta,\delta}$, 
\[
v^+_{\theta, \theta}(x,t)-\etasemi{v}{\Omega \times (0,T)}^\frac{2}{2-\eta}(2\theta)^{\frac{\eta}{2-\eta}}\leq v(x,t)\leq v^-_{\theta, \theta}(x,t)+\etasemi{v}{\Omega \times (0,T)}^\frac{2}{2-\eta}(2\theta)^{\frac{\eta}{2-\eta}}.
\]
\item 
\label{item:inf conv is semiconcave}
In the sense of distributions,  $D^2 v^-_{\theta,\theta}(x,t) \leq \theta^{-1}I$ and $D^2 v^+_{\theta,\theta}(x,t) \geq -\theta^{-1}I$ for every $(x,t)\in U^{\theta,\delta}$.
\item 
\label{item:lip bd on v}
We have $\tlipsemi{v^{\pm}_{\theta,\theta}}{\Omega\times(0,T)}\leq 3T\theta^{-1}$.
\item \label{inf conv is delta soln} If  $v$ is a $\delta$-supersolution  of $v_t-F(D^2v)=0$ in $\Omega \times (0,T)$, then $v^-_{\theta,\theta}$ is a $\delta$-supersolution of $v_t-F(D^2v)=0$ in $U^{\theta,\delta}$. 
 If  $v$ is a $\delta$-subsolution  of $v_t-F(D^2v)=0$ in $\Omega \times (0,T)$, then $v^+_{\theta,\theta}$ is a $\delta$-subsolution of $v_t-F(D^2v)=0$ in $U^{\theta,\delta}$.
\end{enumerate}
\end{prop}

\subsection{Inf and sup convolutions of mesh functions}
In the following proposition we summarize the facts that we need about inf and sup convolutions of mesh functions. 
\begin{prop}
\label{prop: inf-conv for mesh}
Let us take $\theta>0$ and $v\in C^{0,\eta}(\Umesh)$. Then:
\begin{enumerate}
\item 
\label{eq: d(q,p^*)} If $(x^*,t^*)$ is a point at which the infimum (resp. supremum) is achieved in the definition of $\convmesh(x,t)$  (resp. $v^{+}_{\theta, \theta}(x,t) $), then
\[ 
d_e((x^*,t^*),(x,t)) \leq \omega(h,\theta).
\]
In particular, if $(x,t)\in U^h_\theta$ then $(x^*,t^*)\in \Umesh^i$.
\item  We have $v^-_{\theta, \theta}(x,t)\leq v(x,t)\leq v^+_{\theta, \theta}(x,t)$ for all $(x,t)\in \Umesh$.
\item 
\label{item: mesh: bound from below on inf conv}
Assume $(x,t)\in \Omega \times (0,T)$ and $(y,s)\in \Umesh$ is a closest mesh point to $(x,t)$. Then 
\[
v^{+}_{\theta, \theta} (y,s) -\etanorm{v}{\Umesh} \omega(h,\theta)^{\eta}\leq v(x,t)\leq 
\convmesh(y,s) + \etanorm{v}{\Umesh} \omega(h, \theta)^{\eta}.
\]
\item 
\label{item: mesh inf conv is semiconcave}
In the sense of distributions, $D^2 \convmesh(x,t) \leq \theta^{-1}I$  and $D^2 v^{+}_{\theta, \theta} (x,t)\geq -\theta^{-1}I$ for every $(x,t)\in U^h_\theta$.
\item 
\label{item: mesh lip bd on v}
We have $\tlipsemi{v^{\pm}_{\theta,\theta}}{\Omega\times(0,T)}\leq 3T\theta^{-1}$.
\item 
\label{opsandconv}
If $(x,t)\in U^h_\theta$ is such that $Y^-_{Nh}(x,t)\subset U^h_\theta$ and $(x^*,t^*)$ is a point at which the infimum (resp. supremum) is achieved in the definition of $\convmesh(x,t)$  (resp. $v^{+}_{\theta, \theta}(x,t) $), then
\begin{align*}
\delta^-_\tau \convmesh(x,t) \geq \delta^-_\tau v(x^*, t^*) &\text{ and } \delta^2_y\convmesh(x,t)\leq \delta^2_y v(x^*,t^*) \\
(\text{resp. }\delta^-_\tau v^+_{\theta,\theta}(x,t) \leq \delta^-_\tau v(x^*, t^*)  &\text{ and } \delta^2_y v^+_{\theta,\theta}(x,t)\geq \delta^2_y v(x^*,t^*) ).
\end{align*}
\end{enumerate}
\end{prop}
 The proofs of items (1)-(5) are very similar to the elliptic case (see Proposition 2.3 of \cite{Approx schemes} and \cite[Lemma 5.2]{Cabre Caffarelli book}) so we omit them. We provide the proof of item (\ref{opsandconv}).
 \begin{proof}[Proof of item (\ref{opsandconv}) of Proposition \ref{prop: inf-conv for mesh}]
We will give the proof of the  first inequality; the proofs of the others are analogous. By the definition of the discrete operator $\delta^-_\tau$ and the definition of $(x^*,t^*)$, we have
\begin{align*}
\delta^-_\tau \convmesh(x,t) & = \frac{1}{h}\left(\convmesh(x,t)-\convmesh(x,t-h^2)\right)= \frac{1}{h}\left(v(x^*,t^*) + \frac{|x-x^*|^2}{2\theta}+\frac{|t-t^*|^2}{2\theta} -\convmesh(x,t-h^2)\right).
\end{align*}
We now estimate $\convmesh(x,t-h^2)$ from above. Since $(x,t)\in U^h_\theta$,  item  (\ref{eq: d(q,p^*)}) of  this proposition implies $(x^*, t^*) \in \Umesh^i$. In addition, since we have $Y^-_{Nh}(x,t)\subset U^h_\theta$, we find  $(x^*,t^*-h^2)\in \Umesh$. Therefore, we may use $(x^*,t^*-h^2)$ as a ``test point" in the definition of $v^-_{\theta,\theta}(x,t-h^2)$ and obtain,
\begin{align*}
\delta^-_\tau \convmesh(x,t) &\geq \frac{1}{h}\left(v(x^*,t^*) + \frac{|x-x^*|^2}{2\theta}+\frac{|t-t^*|^2}{2\theta} - v(x^*,t^*-h^2) -\frac{|x-x^*|^2}{2\theta}-\frac{|t-t^*|^2}{2\theta}\right)
\\& =\delta^-_\tau v(x^*, t^*),
\end{align*}
where the  equality follows from the definition of the operator $\delta^-_\tau$. 
\end{proof}

\section*{Acknowledgements} The author thanks her thesis advisor, Professor Takis Souganidis, for suggesting this problem and for his guidance and encouragement. She also thanks the anonymous referee, whose comments  have greatly helped her improve the exposition of this paper and clarify many of the arguments.

\end{document}